\documentclass[12pt]{amsart}
\usepackage{amsmath,amssymb,amsthm, xcolor}
\usepackage[margin=1.25in]{geometry}

\colorlet{Mycolor1}{green!10!orange!90!}

\newcommand{\labeltext}[2]{%
	\@bsphack
	\csname phantomsection\endcsname 
	\def\@currentlabel{#1}{\label{#2}}%
	\@esphack
}

\title[Fast Approximation: Beyond the locally constant case]{Fast approximation of Lyapunov exponents: Beyond the locally constant case}
\author{Mark Piraino}
\address{Department of Mathematics \\
	Northwestern University\\
	2033 Sheridan Road Evanston, IL 60208}
\email{mark.piraino@northwestern.edu}
\thanks{M.P. was supported in part by the National Science Foundation grant ``RTG: Analysis on manifolds" at Northwestern University}
\date{\today}

\theoremstyle{definition}
\newtheorem{theorem}{Theorem}[section]
\newtheorem{lemma}[theorem]{Lemma}

\newtheorem{proposition}[theorem]{Proposition}
\newtheorem{definition}[theorem]{Definition}
\newtheorem*{notation}{Notation}
\newtheorem*{remark}{Remark}

\newtheorem*{hyp}{Hypothesis}

\newcommand{\pmat}[1]{\begin{pmatrix} #1 \end{pmatrix}}

\newcommand{\inn}[1]{\left\langle #1 \right\rangle}
\newcommand{\set}[1]{\left\{ #1 \right\}}
\newcommand{\abs}[1]{\left| #1 \right|}

\newcommand{\norm}[1]{\left \| #1 \right \|}

\newcommand{\ol}[1]{\overline{#1}}

\def\[#1\]{\begin{align*}#1\end{align*}}

\newcommand{\floor}[1]{\lfloor #1 \rfloor}

\DeclareMathOperator{\spn}{span}

\DeclareMathOperator{\var}{var}
\DeclareMathOperator{\tr}{tr}
\DeclareMathOperator{\op}{op}

\DeclareMathOperator{\ran}{Ran}
\DeclareMathOperator{\interior}{int}
\DeclareMathOperator{\lip}{Lip}
\DeclareMathOperator{\dist}{dist}
\DeclareMathOperator{\diam}{diam}
\DeclareMathOperator{\tope}{top}
\DeclareMathOperator{\per}{per}

\newcommand{\C}{\mathbb{C}}

\newcommand{\R}{\mathbb{R}}
\newcommand{\RP}{\mathbb{R}\mathbb{P}}

\newcommand{\Z}{\mathbb{Z}}
\def\B{\mathcal{B}}     
     
\def\A{\mathcal{A}}    

\def\H{\mathcal{H}}   
\def\K{\mathcal{K}}             
\def\T{\mathbb{T}}  
\def\L{\mathcal{L}}      
    
\def\S{\Sigma}

\def\T{\mathcal{T}}

\newcommand{\e}{\varepsilon}

\begin{document}
	
\maketitle

\begin{abstract}
	We study the problem of estimating the maximal Lyapunov exponent of dominated cococycles. In particular we are concerned with cocycles over Gibbs states on shifts of finite type for which both the function defining the cocycle and the potential defining the Gibbs state may depend on infinitely many coordinates but are still very regular. We show that when the $n$th variation of both the cocycle and the potential is $O(e^{-cn^{2}})$ for some $c>h_{\text{top}}$ then using periodic points of period less then $n$ the Lyapunov exponent can be approximated to an accuracy $O(n^{-kn})$ for some explicit $k>0$.
\end{abstract}

\section{Introduction}

In 1973 Kingman described the problem of computing Lyapunov exponents as having the pride of place among the unsolved problems of subadditive ergodic theory \cite{MR356192}. While in the intervening years there have been significant results making progress on estimating Lyapunov exponents for random products of positive matrices (\cite{MR2651384}, \cite{MR4014663}, \cite{wang2020pollicotts}) a complete and general theory remains elusive. One direction for which relatively little progress has been made is on computing Lyapunov exponents for cocycles generated by functions which are not locally constant. The goal of this paper is to generalize a well known algorithm for approximating Lyapunov exponents for IID random products of positive matrices due to Pollicott \cite{MR2651384} to this case. Our work generalizes that of Pollicott in two significant directions first our result applies to cocycles generated by functions which are not locally constant and second our result applies to processes which are not IID (and more generally not Markov).

Before we state our results we briefly recall some of the background material we will need. We will work exclusively on one-sided shifts of finite type. So let us recall some notation and definitions. First Given a $0,1$-matrix $T$ we define\
\[ \S_{T}^{+} = \set{(x_{i})_{i=0}^{\infty}: T_{x_{i}x_{i+1}} = 1 \text{ for all }i\geq 0} \]
The usual action on this space is the shift $\sigma((x_{i})_{i=0}^{\infty}) = (x_{i+1})_{i=0}^{\infty}$. It is well known that $\S_{T}^{+}$ can be made into a compact metric space, we will discuss different metrics which we can put on $\S_{T}^{+}$ in section \ref{sec:Bspace}.

Next let us briefly recall a small amount of background on Lyapnuov exponents. Given a shift of finite type $\S_{T}^{+}$ a shift invariant probability measure $\mu$ and a function $\A:\S_{T}^{+} \to GL_{d}(\R)$ we define the maximal Lyapunov exponent of $\A$ over $\mu$ to be the quantity
\[ \gamma_{1}(\A , \mu) = \lim_{n \to \infty}\frac{1}{n}\int_{\S_{T}^{+}} \log \norm{\A(\sigma^{n-1}x)\cdots \A(\sigma x)\A(x)}d\mu. \]
We will also use the notation 
\[ \A^{(n)}(x) = \A(\sigma^{n-1}x)\cdots \A(\sigma x)\A(x). \]
This quantity plays an important role in a number of areas of mathematics. The most prominent being smooth dynamics in which $\A$ is the derivative of smooth map and the maximal Lyapunov exponent describes the rate at which certain orbits diverge from one and other. It is also known that entropy rate of a hidden Markov processes can be expressed as a Lyapunov exponent for a suitable locally constant function $\A$.

We will work with functions $\A$ which are both very regular (more regular then H\"older) and also dominated in the sense of \ref{def:dominated}. For the purpose of the introduction the reader is free to think of $\A$ being dominated as $\A(x)$ being a positive matrix for all $x \in \S_{T}^{+}$. Furthermore our results will apply to measures which are Gibbs states for potentials which are also very regular (but not necessarily locally constant).

By empirical investigation it is known that even in the simplest case in which $\A$ is locally constant and $\mu$ is a Bernoulli measure the quantity 
\[ \frac{1}{n}\int_{\S_{T}^{+}} \log \norm{\A(\sigma^{n-1}x)\cdots \A(\sigma x)\A(x)}d\mu \]
(which can be computed explicitly in this case) converges to $\gamma_{1}(\A, \mu)$ at a very slow rate. The key to Pollicott's algorithm is to use a different characterization of $\gamma_{1}(\A , \mu)$ based on thermodynamic formalism. For a more complete background on thermodynamic formalism we refer the reader to \cite{bowen1975equilibrium} or \cite{MR1085356}. For our purpose it is enough to understand the following intuition. Suppose that we can find a function $\varphi$ for which $\mu$ is the unique shift invariant probability measure which maximizes $h_{\nu}(\sigma) + \int \varphi d \nu$ (here $h_{\nu}(\sigma)$ is the measure theoretic entropy of the shift map) over the space of $\sigma$ invariant probability measures. Such a measure is called an equilibrium state for $\varphi$ (or when $\varphi$ is sufficiently regular, such as in this paper, a Gibbs state). These measures generalize Markov measures in the sense that a measure is Markov exactly when $\varphi$ can be taken locally constant. Provided we have such a $\varphi$ we can define the function
\[ P(\beta) = \sup\set{h_{\nu}(\sigma) + \int \varphi d\nu + \beta \gamma_{1}(\A , \nu) : \nu \text{ shift invaraint probability}} \]
and realize that
\[ P'(0) = \gamma_{1}(\A , \mu). \]
In fact there is other dynamically relevant information contained in $P(\beta)$ such as variances for certain central limit theorems. A similar method has been used to computed other interesting quantities such as Hausdorff dimensions of dynamically relevant sets \cite{MR1855840}, \cite{MR3742587}, \cite{morris2018fast}. It is well known that if $\varphi$ is regular enough and $\A$ satisfies certain properties then function $P(\beta)$ can also be described as $P(\beta) = \log \rho(\L_{\beta})$ where $\rho(\L_{\beta})$ is the spectral radius of the following operator
\[ \L_{\beta}f(x,z) = \sum_{i:ix \in \S_{T}^{+}}e^{\varphi(ix)}\norm{\A(ix)\frac{u}{\norm{u}}}^{\beta}f(ix, \ol{\A(ix)u}). \]
$\L_{\beta}$ naturally acts on function on the set $\S_{T}^{+}\times \RP^{d-1}$. The major technical result that we will prove is the following.

\begin{theorem}\label{thm:eigenvalues}
	Let $\S_{T}^{+}$ be the shift of finite type defined by the matrix $T$, $\varphi:\S_{T}^{+} \to \R$, and $\A:\S_{T}^{+} \to GL_{d}(\R)$ be such that $\A(\S_{T}^{+})$ is dominated and
	\[ \var_{n}\varphi, \var_{n}\A = O(e^{-cn^{p}}). \]
	There is a Banach space $\B$ dense in a quotient of $C(\S_{T}^{+}\times \RP^{d-1})$ such that for any $\beta$ $\L_{\beta}$ acts on $\B$ and the following are true.
	\begin{enumerate}
		\item
		If $p>1$ then $\L_{\beta}$ is compact.
		
		\item 
		If $p=2$ then
		\[ \abs{\lambda_{n}(\L_{\beta})} = O(n^{-k}) \]
		where $\set{\lambda_{n}(\L_{\beta})}$ is the sequence of eigenvalues of $\L_{\beta}$ listed in descending order of modulus with algebraic multiplicity and
		\[ k=\frac{c}{h_{\tope}}-\frac{1}{2}. \]
	\end{enumerate}
\end{theorem}

Having a quantitative bound on the rate of convergence of the eigenvalues of $\L_{\beta}$ to $0$ gives allows us to define a determinant for $\L_{\beta}$ which gives us an implicit description of $\rho(\L_{\beta})$ and hence we can use implicit differentiation to produce a formula for $\gamma_{1}(\A, \mu)$. In addition we can use Theorem \ref{thm:eigenvalues} to bound the Taylor coefficients of this determinant the result of this process is the following theorem which is the main result of this paper.

\begin{theorem}\label{thm:mainintro}
	Let $\S_{T}^{+}$ be the shift of finite type defined by an irreducible matrix $T$, $g:\S_{T}^{+} \to \R$ be a $g$-function, $\A:\S_{T}^{+} \to GL_{d}(\R)$ be such that $\A(\S_{T}^{+})$ is dominated and
	\[ \var_{n}\log g, \var_{n}\A = O(e^{-cn^{2}})\text{ where }c > h_{\tope}. \]
	Then for each $n\geq 1$ there exists an approximation $\gamma_{1}^{(n)}(\A,\mu_{\varphi})$ computable using the values of $g$ and $\A$ at periodic points of period $\leq n$ and
	\[ \abs{\gamma_{1}^{(n)}(\A,\mu_{g}) - \gamma_{1}(\A,\mu_{g})} = O(n^{-nk}) \]
	where $\mu_{g}$ is the unique $g$ measure for $g$ and
	\[ k = \frac{2c - h_{\tope}}{4 h_{\tope}}. \]
\end{theorem}

\begin{remark}
	A $g$-function is a continuous function $g:\S_{T}^{+} \to \R$ such that $g(x)>0$ for all $x \in \S_{T}^{+}$ and
	\[ \sum_{i:ix \in \S_{T}^{+}}g(ix) =1 \]
	for all $x \in \S_{T}^{+}$. For those unfamiliar with $g$-functions we refer the reader to \cite{MR1085356} we comment that the measure $\mu_{g}$ in the previous theorem is also the unique Gibbs state for $\log g$. We also remark that in practice we only use the fact that $\mu_{g}$ is a $g$ measure to ensure that $P(0)=0$. In practice we require only that we are able to compute $P(0)$.
\end{remark}

The paper is organized in the following way. In section \ref{sec:prelim} we collect some results from the literature that we will need in our proof, in section \ref{sec:Bspace} we define the space $\B$ in Theorem \ref{thm:eigenvalues} and show that $\L_{\beta}$ acts on it, in section \ref{sec:approx} we prove Theorem \ref{thm:eigenvalues}, in section \ref{sec:traceformula} we establish a formula for the trace of $\L_{\beta}$, in section \ref{sec:determinant} we prove Theorem \ref{thm:mainintro}. A number of technical results are postponed to the appendix (Section \ref{sec:appendix}) so that they do not unnecessarily interrupt the narrative follow of the proofs.

\section{Preliminaries}\label{sec:prelim}

\begin{definition}\label{def:dominated}
	We will say that a set $\mathrm{A} \subset M_{d}(\R)$ is \emph{dominated} if there exists $C, k >0$ such that 
	\[ \sup_{A_{1},A_{2},\cdots A_{n} \in \mathrm{A}}\frac{\sigma_{2}(A_{n}\cdots A_{2}A_{1})}{\sigma_{1}(A_{n}\cdots A_{2}A_{1})} \leq C e^{-k n} \]
	for all $n \geq 1$. Where $\sigma_{1}(A),\sigma_{2}(A)$ are the first and second singular values of $A$ respectively. 
\end{definition}

\begin{definition}
	Let $\mathsf{A} \subseteq M_{d}(\R)$ be nonempty. We say that $(\K_{1},\ldots, \K_{m})$ is a \emph{multicone} for $A$ if the following properties hold. 
	\begin{enumerate}
		\item 
		Each $\K_{i}$ is a closed, convex subset of $\R^{d}$ with nonempty interior such that $\lambda \K_{i} \subseteq \K_{i}$ for all $\lambda \geq 0$.
		
		\item 
		There exists a unit vector $w \in \R^{d}$ such that $\inn{u,w}>0$ for all nonzero vectors $v \in \bigcup_{i}\K_{i}$.
		
		\item 
		For all $A \in \mathsf{A}$ and $j \in \set{1,\ldots, m}$ there exists $\ell = \ell(j, A) \in \set{1, \ldots, m}$ such that $A(K_{i}\setminus \set{0})\subseteq (\interior \K_{\ell}) \cup (-\interior \K_{\ell})$.
		
		\item 
		For all $i,j \in \set{1, \ldots, m}$ with $i\neq j$ $\K_{i} \cap \K_{j} = \set{0}$.
	\end{enumerate}
	The vector $w$ is called the \emph{transverse defining vector} for $(\K_{1}, \ldots, \K_{m})$. If there exists a multicone for $\mathsf{A}$ then $\mathsf{A}$ is said to be \emph{multipositive}.
\end{definition}

It is known that a compact set $\mathrm{A} \subset GL_{d}(\R)$ is dominated if and only if it is multipositive for a proof we refer the reader to \cite{morris2018fast}. Given a compact multipositive set we can extend the projective action of the matrices so that they map an open set in $\C$ strictly inside itself.

\begin{theorem}[Morris \cite{morris2018fast}]\label{thm:conesmorris}
	Let $d \geq 1$, let $\mathsf{A} \subseteq M_{d}{(\R)}$ be compact and nonempty and suppose that $(\K_{1}, \ldots,. \K_{m})$ is a multicone for $\mathsf{A}$ with transverse defining vector $w$. For each $j = 1, \ldots, m$ define
	\[ \K_{j}^{\C}:= \set{\lambda ((u+v)+i(u-v)) : \lambda  \in \C \text{ and }u,v \in\K_{j}}, \]
	and let
	\[ U_{i} :=\set{z \in \C^{d}:z \in \K_{j}^{\C} \text{ and }\inn{z,w}=1}\text{ and }U:= \bigcup_{j=1}^{m}U_{j} \]
	For each $A \in \mathsf{A}$ and $z \in U$ write $\ol{A}z : = \inn{Az,w}^{-1}Az$. Then
	\begin{enumerate}
		\item 
		Every $A \in \mathcal{S}(\mathsf{A})$ has a simple leading eigenvalue $\lambda_{1}(A)$ which is real and strictly larger in modulus than all of the other eigenvalues of $A$.
		
		\item 
		There is a constant $\tau>0$ such that $\norm{A_{1}A_{2}} \geq \tau \norm{A_{1}}\norm{A_{2}}$ for all $A_{1},A_{2}\in \mathcal{S}(\mathsf{A})$.
		
		\item 
		$U$ is a nonempty, relatively open, bounded subset of the complex hyperplane $\set{x \in \C^{d}:\inn{z,w}=1}$, and for every $A \in \mathcal{S}(\mathsf{A})$ the math $\ol{A}:U \to U$ is well defined.
		
		\item 
		There exist constants $C,k>0$ such that 
		\[ \sup_{A_{1},A_{2}, \ldots, A_{n} \in \mathsf{A}}\diam\ol{A_{1}A_{2}\cdots A_{n}}(U) \leq C e^{-k n} \]
		for all $n\geq 1$.
		
		\item 
		For each $A \in \mathcal{S}(\mathsf{A})$ the map $\ol{A}$ has a unique fixed point $z_{A} \in U$. We have that $z_{A}$ is an eigenvector for the eigenvalue $\lambda_{1}(A)$. The eigenvalues of the derivative $D_{z_{A}}\ol{A}$ are precisely the numbers $\lambda_{j}(A)/\lambda_{1}(A)$ for $j=2, \ldots, d$, and in particular
		\[ \det(I-D_{z_{A}}\ol{A})=\frac{p_{A}'(\lambda_{1}(A))}{\lambda_{1}(A)^{d-1}}\neq 0 \]
		where $p_{A}(x) = \det(I-xA)$ is the characteristic polynomial and $p'(x)$ is its first derivative.
		
		\item \label{thm:conesmorris/conenorminequality}
		There  is a constant $C>0$ such that for each $A \in \mathcal{S}(\mathsf{A})$ and $z \in U$ we have that 
		\[  C^{-1}\norm{A}\leq \abs{\inn{Az,w}} \leq C \norm{A}. \]
		
		\item 
		The set $\bigcup_{A \in \mathcal{S}(\mathsf{A})}\ol{A}(U)$ is compactly contained in $U$.
		
		\item 
		The collection $\set{U_{j}}_{j=1}^{m}$ is disjoint.
	\end{enumerate}
\end{theorem}

Take $K \subseteq U$ to be a compact set with $\bigcup_{A \in \mathcal{S}(\mathsf{A})}\ol{A}(U) \subseteq K$.

\begin{proposition}\label{prop:locallyconstantcones}
	Suppose that $\A:\S_{T}^{+}\to M_{d}(\R)$ is continuous and that $\A(\S_{T}^{+})$ is mulitpositive with multicone $(\K_{1}, \ldots, \K_{m})$. There exists an $N$ such that for any word $I$ with $\abs{I}=N$ and $i \in \set{1, \ldots , m}$ there exists a number $\ell = \ell(i,I) \in \set{1, \ldots, m}$ such that $\ol{\A(x)}U_{j} \subseteq U_{\ell}$ for all $x \in [I]$.
\end{proposition}
\begin{proof}
	As $\A$ is continuous so is $\ol{\A(x)}$. Take $N$ large enough such that if $x_{i}=y_{i}$ for all $0 \leq i \leq N-1$ then $\norm{\ol{\A(x)} - \ol{\A(y)}} < \min_{j_{1},j_{2}} \dist(U_{j_{1}}, U_{j_{2}})$.
\end{proof}

\begin{remark}
	Up to taking a higher block representation we may assume that $N=1$.
\end{remark}

\begin{hyp}
	We will say that a function $\A:\S_{T}^{+} \to GL_{d}(\R)$ satisfies \textbf{(H1)} if $\A$ is continuous, $\A(\S_{T}^{+})$ is multipositive and satisfies the conditions of Proposition \ref{prop:locallyconstantcones} with $N = 1$.
\end{hyp}

If we assume that $\A$ satisfies \textbf{(H1)} and take $U$ as in Theorem \ref{thm:conesmorris} then for each $x \in \S_{T}^{+}$ the operator which does a precomposition by  $\ol{\A(x)}$ acts on the Bergman space $A^{2}(U)$. For more detail on Bergman spaces we refer the reader to \cite{MR2414325}. We briefly recall the definition of $A^{2}(U)$.
\[A^{2}(U)=\set{f : f \text{ is analytic on }U \text{ and }\norm{f}_{A^{2}(U)}<\infty}\]
where
\[ \norm{f}_{A^{2}(U)} = \left(\int_{U}\abs{f(x+iy)}^{2}dxdy\right)^{1/2}. \]
It is well known that $A^{2}(U)$ is a Hilbert space with the $L^{2}$ inner product.

\section{The operators $\L_{\beta}$ and the Banach space $\B$}\label{sec:Bspace}

Let $\varphi:\S_{T}^{+}\to \R$ define
\[ \var_{n}\varphi = \sup\set{\abs{\varphi(x)-\varphi(y)}:x_{i}=y_{i} \text{ for all }0\leq i \leq n-1}. \]
Similarly let $\A:\S_{T}^{+} \to M_{d}(\R)$ and define
\[ \var_{n}\A =\sup\set{\norm{\A(x)-\A(y)}:x_{i}=y_{i} \text{ for all }0\leq i \leq n-1}. \]

\begin{proposition}
	Given a sequence $\set{\kappa_{n}}_{n=1}^{\infty}$ decreasing to $0$ the function
	\[ d_{\set{\kappa_{n}}}(x,y)=\kappa_{k(x,y)} \]
	where $k(x,y)=\min\set{i:x_{i}\neq y_{i}}$ defines an ultrametric on $\S_{T}^{+}$.
\end{proposition}

\begin{hyp}
	We will say that $\varphi$ and $\A$ satisfy \textbf{H2}$(\set{\kappa_{n}})$ if
	\[ \var_{n}\varphi, \var_{n}\A = O(\kappa_{n}). \]
\end{hyp}

Define the space of Lipschitz function in the $d_{\set{\kappa_{n}}}$ metric
\[ \lip(\S_{T}^{+}, \set{\kappa_{n}})=\set{f \in C(\S_{T}^{+}\to \C): \sup_{k}\kappa_{k}^{-1}\var_{k}f<\infty} \]
This becomes a Banach space in the usual way. Define
\[ \abs{f}_{\set{\kappa_{n}}} = \sup_{k\geq 1}\kappa_{k}^{-1}\var_{k}f \]
and set 
\[ \norm{f}_{\lip(\S_{T}^{+}, \set{\kappa_{n}})}=\norm{f}_{\infty}+\abs{f}_{\set{\kappa_{n}}}.  \]

 Let $f(x,z)$ be a function on  $\S_{T}^{+}\times U$. Define
\[ \norm{f}_{\B(U, \kappa_{n})} = \sup_{z\in U} \norm{f(\cdot ,z)}_{\lip(\S_{T}^{+},\set{\kappa_{n}})}. \]
Let
\[ \B(U, \kappa_{n}) =\set{ f(x,z) : z \mapsto f(x,z) \text{ is analytic on $U$ for all }x \in \S_{T}^{+} \text{ and }\norm{f}_{\B(U, \kappa_{n})}<\infty}.  \]
When it is clear from context what $U$ or $\kappa_{n}$ are we will right $\B(U,\kappa_{n}) = \B(U) = \B(\kappa_{n}) = \B$. In addition it will be convenient for us to define the set $B(U,\cdot)$ to the space of functions such that $z \mapsto f(x, z)$ is analytic on $U$ for all $x \ in \S_{T}^{+}$ and $\set{f(\cdot , z): z \in U}$ is equicontinuous.

\begin{proposition}
	For any open set $U \subseteq \C$ and decreasing sequence converging  to $0$ $\set{\kappa_{n}}$ $(\B(U , \kappa_{n}) , \norm{\cdot}_{\B(U, \kappa_{n})})$ is a Banach space.
\end{proposition}
\begin{proof}
	The ``$\infty$" direct sum of Banach spaces is a Banach space (see for instance \cite[III Proposition 4.4]{MR1070713}). That is to say that
	\[ \bigoplus_{z \in U}\lip(\S_{T}^{+},\set{\kappa_{n}}) = \set{ f(x,z) : \sup_{z\in U} \norm{f(\cdot ,z)}_{\lip(\S_{T}^{+},\set{\kappa_{n}})} < \infty}  \]
	with norm $\sup_{z\in U} \norm{f(\cdot ,z)}_{\lip(\S_{T}^{+},\set{\kappa_{n}})}$ is a Banach space. The set $\B(U, \kappa_{n})$ is a closed subspace and hence a Banach space.
\end{proof}

Notice that we can view $\B(U,\kappa_{n})$ as a subset of $C\left(\S_{T}^{+} \times \bigcup_{i}\mathcal{K}_{i}\right)$ (which is a quotient of $C(\S_{T}^{+} \times \RP^{d-1}))$ by taking the quotient $\B(U,\kappa_{n})/\Lambda$ where
\[ \Lambda = \set{f \in \B(U,\kappa_{n}): f(x, z) = 0 \text{ for all }z \in \R^{d}}. \]
In this case the quotient map is actually an isomorphism because analytic functions are entirely determined by their values on $\R^{d}$ then using the Stone-Weierstrass theorem it can be verified that $\B(U, \kappa_{n})$ is dense in $C\left(\S_{T}^{+} \times \bigcup_{i}\mathcal{K}_{i}\right)$.

If we assume that $\A$ satisfies \textbf{H1} and let $U$ and $\ol{\A(x)}$ be as in Theorem \ref{thm:conesmorris} then we can define the operator
\[ \L_{\beta}f(x,z) = \sum_{i:ix \in \S_{T}^{+}}e^{\varphi(ix)}\norm{\A(ix)\frac{z}{\norm{z}}}^{\beta}f(ix, \ol{\A(ix)z}). \]
Where $f(x,z)$ is a function on $\S_{T}^{+} \times U$. Notice that $z \mapsto \norm{\A(ix)\frac{z}{\norm{z}}}^{\beta}$ and $z \mapsto \ol{\A(ix)z}$ are analytic. Thus if $f \in \B(U, \kappa_{n})$ then $\L_{\beta}f(x,z)$ is a linear combination of analytic functions and hence $z \mapsto \L_{\beta}f(x,z)$ is analytic. So to show that $\L_{\beta}$ acts continuously on $\B(U, \kappa_{n})$ (for a suitable $\kappa_{n}$) then we must primarily be concerned with the Lipschitz constants of the functions $\L_{\beta}f(\cdot , z)$.

\begin{lemma}\label{lem:variationinequality}
	Let $\beta \in \R$, assume that $\varphi$ is continuous, $\A$ satisfies \textbf{H1} and let $K$ and $U$ be as in Theorem \ref{thm:conesmorris}. For any open set $W$ with $K \subset W \subset U$ there exists a constant $C>0$ such that for all $z \in U$ we have that
	\[ &\var_{k}\L_{\beta}f(\cdot , z) \\
	&\leq C \left(\sup_{x}\norm{f(x,\cdot)}_{A^{2}(W)} \var_{k+1}\A + \sup_{z \in K}\var_{k+1}f(\cdot, z) + \sup_{x}\norm{f(x,\cdot)}_{A^{2}(W)}\var_{k+1}\varphi \right) \]
	for all $k \geq 1$ and $f \in B(U, \cdot)$.
\end{lemma}
\begin{proof}
	Let $k \geq 1$. Let $z \in U$ and $x,y\in \S_{T}^{+}$ with $x_{i}=y_{i}$ $0\leq i \leq k-1$ then 
	\begin{align*}
	&\abs{\L_{\beta} f(x,z)- \L_{\beta} f(y,z)}\\
	& = \abs{\sum_{i:ix \in \S_{T}^{+}} e^{\varphi(ix)}\norm{\A(ix)\frac{z}{\norm{z}}}^{\beta}f(ix,\ol{\A(ix)}z)- e^{\varphi(iy)}\norm{\A(iy)\frac{z}{\norm{z}}}^{\beta}f(iy,\ol{\A(iy)}z)}\\
	& \leq \sum_{i:ix \in \S_{T}^{+}}\abs{ e^{\varphi(ix)}\norm{\A(ix)\frac{z}{\norm{z}}}^{\beta}f(ix,\ol{\A(ix)}z)- e^{\varphi(iy)}\norm{\A(iy)\frac{z}{\norm{z}}}^{\beta}f(iy,\ol{\A(iy)}z)}.
	\end{align*}
	For each $i$ the quantity
	\begin{equation}\label{eq:terminsum}
	\abs{ e^{\varphi(ix)}\norm{\A(ix)\frac{z}{\norm{z}}}^{\beta}f(ix,\ol{\A(ix)}z)- e^{\varphi(iy)}\norm{\A(iy)\frac{z}{\norm{z}}}^{\beta}f(iy,\ol{\A(iy)}z)} 
	\end{equation}
	is bounded by the sum of 
	\begin{eqnarray}
	e^{\varphi(ix)}\norm{\A(ix)\frac{z}{\norm{z}}}^{\beta}\abs{f(ix,\ol{\A(ix)}z)-f(ix,\ol{\A(iy)}z)}, \label{eq:firstterm}\\
	e^{\varphi(ix)}\norm{\A(ix)\frac{z}{\norm{z}}}^{\beta}\abs{f(ix,\ol{\A(iy)}z)  - f(iy,\ol{\A(iy)}z)}, \label{eq:secondterm}\\
	\abs{f(iy,\ol{\A(iy)} z)}\norm{\A(ix)\frac{z}{\norm{z}}}^{\beta}\abs{ e^{\varphi(ix)}- e^{\varphi(iy)}}, \text{ and } \label{eq:thirdterm}\\
	\abs{f(iy,\ol{\A(iy)} z)}e^{\varphi(iy)}\abs{\norm{\A(ix)\frac{z}{\norm{z}}}^{\beta} - \norm{\A(iy)\frac{z}{\norm{z}}}^{\beta}}. \label{eq:fourthterm}
	\end{eqnarray}
	First lets bound equation \eqref{eq:firstterm}. By Lemma \ref{lem:analyticinequalities} we have that
	\begin{align*}
	\abs{f(ix,\ol{\A(ix)}z)-f(ix,\ol{\A(iy)}z)}&\leq C_{K,W}\norm{f(ix, \cdot)}_{A^{2}(W)}\norm{\ol{\A(ix)}z-\ol{\A(iy)}z}\\
	& \leq C_{K,\A,W}\sup_{x}\norm{f(x,\cdot)}_{A^{2}(W)}\var_{k+1}\A \text{  (by lemma \ref{lem:techbounds})}
	\end{align*}
	
	To bound equation \eqref{eq:secondterm} notice that
	\begin{align*}
	\abs{f(ix,\ol{\A(iy)}z)  - f(iy,\ol{\A(iy)}z)}&\leq \var_{k+1}f(\cdot , \ol{\A(iy)}z)\\
	& \leq \sup_{z \in K} \var_{k+1}f(\cdot , z).
	\end{align*}
	To bound equation \eqref{eq:thirdterm} notice that
	\begin{align*}
	\abs{ e^{\varphi(ix)}- e^{\varphi(iy)}}&\leq e^{\norm{\varphi}_{\infty}} \abs{\varphi(ix) - \varphi(iy)}\\
	&\leq \var_{k+1}\varphi.
	\end{align*}
	By Lemma \ref{lem:techbounds} we can bound \eqref{eq:fourthterm} by
	\[ C_{K,\A,W,\beta}\sup_{x}\norm{f(x, \cdot)}_{A^{2}(W)} e^{\norm{\varphi}_{\infty}}\var_{k+1}\A  \]
	Hence the result.
\end{proof}

\begin{lemma}\label{lem:uniform_norm}
	Assume that $\varphi$ is continuous and $\A$ satisfies \textbf{H1} and let $K$ and $U$ be as in Theorem \ref{thm:conesmorris}. For any open set $W$ with $K \subseteq W \subseteq U$ define the operator $J : A^{2}(U) \to A^{2}(W)$ by $Jf = f|_{W}$. There is a constant $C>0$ such that 
	\[ \abs{\L_{\beta}f(x,z)}\leq C \sup_{x \in \S_{T}^{+}}\norm{Jf(x,\cdot)}_{A^{2}(W)} \]
	for all $x \in \S_{T}^{+}$, $z \in U$ and $f\in \B(U , \cdot)$.
\end{lemma}
\begin{proof}
	We have that 
	\begin{align*}
	\abs{\L_{\beta} f(x,z)} &\leq \sum_{i:ix \in \S_{T}^{+}} e^{\varphi(ix)}\norm{\A(ix)\frac{z}{\norm{z}}}^{\beta}\abs{f(ix,\ol{\A(ix)}z)}\\
	&=\sum_{i:ix \in \S_{T}^{+}} e^{\varphi(ix)}\norm{\A(ix)\frac{z}{\norm{z}}}^{\beta}\abs{Jf(ix,\ol{\A(ix)}z)}\\
	&\leq \abs{\S} e^{\norm{\varphi}_{\infty}}\max\set{(\sup_{x}\norm{\A(x)})^{\beta}, (\inf_{x}\sigma_{d}(\A(x)))^{\beta}}C_{K,W}\norm{Jf(ix, \cdot)}_{A^{2}(W)}.
	\end{align*}
\end{proof}

\begin{remark}
	Notice that from the proofs the constants in the Lemma \ref{lem:uniform_norm} and Lemma \ref{lem:variationinequality} are continuous in $\beta$ and hence can be taken uniform over the set $\beta \leq 1$. 
\end{remark}

\begin{proposition}
	Assume that $\A$ satisfies \textbf{H1} and that $\varphi, \A$ satisfy \textbf{H2}$(\set{\kappa_{n}})$. The associated operator $\L_{\beta}:\B(U, \kappa_{n}) \to \B(U, \kappa_{n})$ is bounded.
\end{proposition}
\begin{proof}
	Let $f \in \B(U, \kappa_{n})$. By Lemma \ref{lem:variationinequality} we have that there is a constant $C_{1}>0$ such that
	\[ \sup_{z\in U}\kappa_{k}^{-1}\var_{k}\L_{\beta}f(\cdot , z)\leq C_{1} \norm{f}_{\B(U, \kappa_{n})} \]\
	By Lemma \ref{lem:uniform_norm} we have that there is a constant $C_{2}>0$ such that 
	\[ \abs{\L_{\beta} f(x,z)}\leq C_{2} \sup_{x \in \S_{T}^{+}}\norm{f(x,\cdot)}_{A^{2}(U)}\leq C_{2} \norm{f}_{\B(U,\kappa_{n})} \]
	for all $z \in U$ and $x \in \S_{T}^{+}$. Thus we have that 
	\[ \sup_{z \in U}\norm{\L_{\beta} f(\cdot , z)}_{\lip(\S_{T}^{+}, \set{\kappa_{n}})}\leq (C_{1}+C_{2})\norm{f}_{\B(U,\kappa_{n})}. \]
	Therefore $\L_{\beta}$ is bounded.
\end{proof}

\section{Approximation numbers and eigenvalues}\label{sec:approx}

Throughout this section we will assume that $\A$ satisfies \textbf{H1} and take $K$ and $U$ as in Theorem \ref{thm:conesmorris}. We will also assume that $\varphi, \A$ satisfy \textbf{H2}$(\set{\kappa_{n}})$ for some $\set{\kappa_{n}}$ and denote $\B(U, \kappa_{n}) = \B(U)$ and similarly for $\B(W)$. We will use a method similar to \cite{MR4147353}.

Recall that the \emph{approximation numbers} of a linear operator $L$ is the sequence
\[ a_{n}(L)=\inf \set{\norm{L-F}: \ran(F)<n}. \]
where $\ran(F)$ is the rank of $F$.

\begin{proposition}\label{prop:approxnum}
	\begin{enumerate}
		\item 
		For any $n,m \geq 1$ and $S,T$ bounded linear operators
		\[ a_{n+m-1}(S+T) \leq a_{m}(T)+a_{n}(S). \]
		
		\item 
		For any $n \geq 1$ and $S_{1},T,S_{2}$ bounded linear operators
		\[ a_{n}(S_{1}TS_{2}) \leq \norm{S_{1}}a_{n}(T)\norm{S_{2}}. \]
	\end{enumerate}
\end{proposition}
For a compact linear operator $\L$ let $\set{\lambda_{n}(\L)}$ be the eigenvalues of $\L$ listed with algebraic multiplicity in non-increasing order by modulus.

\begin{theorem}[Weyl's Inequality]\label{thm:weyl}
	Let $T$ be a compact linear operator on a Banach space. Then
	\[ \prod_{k=1}^{n}\abs{\lambda_{k}(T)} \leq n^{n/2}\prod_{k=1}^{n}a_{k}(T). \]
\end{theorem}

For a Hilbert space the inequality holds without the $n^{n/2}$ factor, it is known that for a general Banach space the term $n^{n/2}$ is optimal \cite{MR2180891}. 

\begin{theorem}
	Let $\H$ be a Hilbert space and $T:\H \to \H$ be a compact operator. Then $T$ has an expansion 
	\[ T=\sum_{k=1}^{N}s_{k}(T)\inn{\phi_{k} , \cdot}\psi_{k} \]
	where $N$ is either a non-negative integer or $\infty$, $s_{n}(T)$ are the singular values of $T$, and $\set{\phi_{n}}$ and $\set{\psi_{n}}$ are orthonormal sets. 
\end{theorem}

\begin{theorem}[Bandtlow-Jenkinson \cite{MR2414325}]
	Suppose that $U_{1},U_{2}$ are open subsets of $\C^{d}$ with $U_{2}$ compactly contained in $U_{1}$ define the operator $J:A^{2}(U_{1})\to A^{2}(U_{2})$ by
	\[ Jf = f|_{U_{2}}. \]
	Then $J$ is compact and there exist constants $C, c >0$ such that 
	\[  s_{n}(J)\leq Ce^{-cn^{1/d}}. \]
\end{theorem}

\begin{proposition}
	Let $J_{n} = \sum_{k=1}^{n}s_{k}(J)\inn{\phi_{k} , \cdot}\psi_{k} $. Then there exist constants $C,c_{\A}>0$
	\[ \norm{J_{n}-J}\leq Ce^{-c_{\A}n^{1/d}}. \]
\end{proposition}
\begin{proof}
	Notice that
	\begin{align*}
	\norm{J_{n}-J} &= \norm{\sum_{k=n}^{\infty}s_{k}(T)\inn{\phi_{k} , \cdot}\psi_{k}}\\
	&\leq \sum_{k=n}^{\infty}s_{k}(T)\\
	&\leq C \sum_{k=n}^{\infty}e^{-ck^{1/d}}.
	\end{align*}
	Consider
	\begin{align*}
	\sum_{k=n}^{\infty}e^{-c_{\A}k^{1/d}}& \leq e^{-c n^{1/d}}+\int_{n}^{\infty}e^{-cx^{1/d}}dx\\
	&= e^{-c n^{1/d}} + d a^{-d} \Gamma(d, cn^{1/d})\\
	&= e^{-c n^{1/d}} + d a^{-d} (d-1)! e^{-cn^{1/d}}\sum_{k=0}^{d-1}\frac{(cn^{1/d})^{k}}{k!}\\
	&= O(e^{-c_{\A}n^{1/d}})
	\end{align*}
	where $c_{\A}<c$.
\end{proof}

\begin{proposition}\label{prop:finite_dim_projections}
	For each allowed word, $I$, of length $m$ pick a point $x_{I}$ in the cylinder set defined by the word $I$, $[I]$. Define an operator $E_{n,m}:\B(U)\to \B(W)$ by
	\[ E_{n,m}f(x,z) = \sum_{\abs{I}=m}J_{n}f(x_{I},z)\chi_{[I]}(x). \]
	Then the following are true:
	\begin{enumerate}
		\item 
		$\dim \ran(E_{n,m})\leq n \cdot \abs{L_{m}}$.
		
		\item
		For all $x \in \S_{T}^{+}$ and $f \in \B(U)$ we have that
		\[ \norm{Jf(x,\cdot)-E_{n,m}f(x,\cdot)}_{A^{2}(W)} &\leq \norm{J}_{A^{2}(U)\to A^{2}(W)}\sup_{z \in U}\var_{m}f(\cdot ,z) \\ &\;\;\;\;\;\;\;\;+ \norm{J-J_{n}}_{A^{2}(U) \to A^{2}(W)} \sup_{x}\norm{f(x,\cdot)}_{A^{2}(U)}. \]
		
		\item 
		There is a constant $C>0$ such that
		\[ \var_{k}(Jf-E_{n,m}f)(\cdot , z)\leq C\sup_{z \in U}\var_{k}f(\cdot , z) \]
		for any $z \in K$ and $f \in \B(U)$.
		
		\item \label{prop:finite_dim_projections/finalvarbound}
		There is a constant $C>0$ such that 
		\[ \var_{k}(Jf-E_{n,m}f)(\cdot , z)\leq C\min \set{\kappa_{k}, \kappa_{m} + \norm{J-J_{n}}_{A^{2}(U)\to A^{2}(W)} }. \]
		for all $z \in K$ and $f \in \B(U)$ with $\norm{f}_{\B(U)}\leq 1$.
	\end{enumerate}
\end{proposition}
\begin{proof}
	\begin{enumerate}
		\item 
		Notice that $E_{n,m}f$ is contained in
		\[ \spn \set{\psi_{k}(z)\chi_{I}(x) : \abs{I}=m, 1 \leq k \leq n}. \]
		
		\item 
		Let $x \in \S_{T}^{+}$ and notice that
		\begin{align*}
		&\norm{Jf(x,\cdot)-E_{n,m}f(x,\cdot)}_{A^{2}(W)} \\
		&= \norm{Jf(x,\cdot)- J_{n}f(x_{x_{0}\cdots x_{m-1}}, \cdot)}_{A^{2}(W)}\\
		&\leq \norm{Jf(x,\cdot)- Jf(x_{x_{0}\cdots x_{m-1}}, \cdot)}_{A^{2}(W)} + \norm{Jf(x_{x_{0}\cdots x_{m-1}}, \cdot)-J_{n}f(x_{x_{0}\cdots x_{m-1}}, \cdot)}_{A^{2}(W)}\\
		&\leq \norm{J}_{A^{2}(U)\to A^{2}(W)}\sup_{z \in U}\var_{m}f(\cdot ,z) + \norm{J-J_{n}}_{A^{2}(U) \to A^{2}(W)} \sup_{x}\norm{f(x,\cdot)}_{A^{2}(U)}.
		\end{align*}	
		
		\item 
		Suppose that $x,y \in \S_{T}^{+}$ and $x_{i}=y_{i}$ for all $0 \leq i \leq k-1$. Notice that
		\begin{align*}
		&\abs{f(x,z)-E_{n,m}f(x,z)-f(y,z)+E_{n,m}f(y,z)}\\
		&\leq \abs{f(x,z)-f(y,z)}+\abs{E_{n,m}f(x,z)-E_{n,m}f(y,z)}\\
		&\leq \var_{k}f(\cdot ,z)  + \abs{E_{n,m}f(x,z)-E_{n,m}f(y,z)}.
		\end{align*}
		If $k \geq m$ then 
		\[ E_{n,m}f(x,z)=J_{n}f(x_{x_{0} \cdots x_{m-1}} , z)=J_{n}f(x_{y_{0}\cdots y_{m-1}} , z) =E_{n,m}f(y,z).\]
		If $k<m$ then 
		\begin{align*}
		\abs{E_{n,m}f(x,z)-E_{n,m}f(y,z)}&= \abs{J_{n}f(x_{x_{0}\cdots x_{m-1}} , z) - J_{n}f(x_{y_{0}\cdots y_{m-1}} , z)}\\
		& \leq C_{K}\norm{J_{n}f(x_{x_{0}\cdots x_{m-1}} , z) - J_{n}f(x_{y_{0}\cdots y_{m-1}} , z)}_{A^{2}(W)}\\
		&\leq C_{K}\norm{J_{n}}_{A^{2}(U)\to A^{2}(W)}\norm{f(x_{x_{0}\cdots x_{m-1}} , \cdot)-f(x_{y_{0}\cdots y_{m-1}}, \cdot)}_{A^{2}(U)}\\
		&\leq C_{K}\norm{J_{n}}_{A^{2}(U)\to A^{2}(W)} \sup_{z \in U}\var_{k}f(\cdot , z)\\
		&\leq C_{K}\sum_{k=1}^{n}s_{k}(J) \sup_{z \in U}\var_{k}f(\cdot , z) \\
		&\leq C_{K}\sum_{k=1}^{\infty}s_{k}(J) \sup_{z \in U}\var_{k}f(\cdot , z).
		\end{align*}
		
		\item 
		Let $f \in \B(U)$ with $\norm{f}_{\B(U)}\leq 1$. Notice that for any $z \in K$ we have that
		\[ \var_{k}(Jf-E_{n,m}f)(\cdot , z)\leq 2\sup_{z \in K}\norm{Jf(\cdot, z)-E_{n,m}f(\cdot , z)}_{\infty}. \]
		There is a constant $C_{K}$ such that for any $x \in \S_{T}^{+}$ we have
		\begin{align*}
		\abs{Jf(x,z)-E_{n,m}f(x,z)} &\leq C_{K} \norm{Jf(x, \cdot) - E_{n,m}f(x,\cdot)}_{A^{2}(W)}\\
		& \leq  \norm{J}_{A^{2}(U)\to A^{2}(W)}\sup_{z \in U}\var_{m}f(\cdot ,z) \\ 
		&\;\;\;\;\;\;\;\;+ \norm{J-J_{n}}_{A^{2}(U) \to A^{2}(W)} \sup_{x}\norm{f(x,\cdot)}_{A^{2}(U)}\\
		& \leq \norm{J}_{A^{2}(U)\to A^{2}(W)}\kappa_{m} + e^{-c_{\A} n^{1/d}}.
		\end{align*}
		On the other hand
		\begin{align*}
		\var_{k}(Jf-E_{n,m}f)(\cdot , z)& \leq \var_{k}Jf(\cdot, z)+\var_{k}E_{n,m}f(\cdot , z)\\
		&\leq 2 \var_{k}f(\cdot,z).
		\end{align*} 
	\end{enumerate}
\end{proof}

\begin{lemma}\label{lem:lipbound}
	Let $\beta \in \R$. There exists a constant $C>0$ such that
	\[ \abs{\L_{\beta}f(\cdot , z) - \L_{\beta}E_{n,m}f(\cdot , z)}_{\set{\kappa_{k}}} \leq C\norm{f}_{\B(U)}\left(\max\set{\sup_{k\geq 1}\min\set{\frac{\kappa_{k+1}}{\kappa_{k}}, \frac{\kappa_{m}+e^{-c_{\A}n^{1/d}}}{\kappa_{k}} } , \kappa_{m}+e^{-c_{\A}n^{1/d}} }\right) \]
	for all $z \in U$.
\end{lemma}
\begin{proof}
	Let $z \in U$. Notice that
	\begin{align*}
	\abs{\L_{\beta}f(\cdot , z) - \L_{\beta}E_{n,m}f(\cdot , z)}_{\set{\kappa_{k}}} =\abs{\L_{\beta}(f(\cdot , z) - E_{n,m}f(\cdot , z))}_{\set{\kappa_{k}}}
	\end{align*}
	Thus we must bound
	\begin{equation}\label{eq:variationterm}
	\kappa_{k}^{-1}\var_{k}\L_{\beta}(f(\cdot , z)-E_{n,m}f(\cdot , z))
	\end{equation}
	and by Lemma \ref{lem:variationinequality} we have that for any $k\geq 1$ there is a constant $C>0$ such that equation \eqref{eq:variationterm} is bounded above by the sum of 
	\begin{eqnarray}
	C\sup_{x}\norm{Jf(x , \cdot)-E_{n,m}f( x, \cdot)}_{A^{2}(W)} \kappa_{k}^{-1}\var_{k+1}\A \label{eq:variation_sum_1}\\
	C\sup_{z \in K}\kappa_{k}^{-1}\var_{k+1}(Jf-E_{n,m}f)(\cdot, z)\label{eq:variation_sum_2} \\
	C\sup_{x}\norm{Jf(x , \cdot)-E_{n,m}f( x, \cdot)}_{A^{2}(W)}\kappa_{k}^{-1}\var_{k+1}\varphi.\label{eq:variation_sum_3}
	\end{eqnarray}
	Notice that by Proposition \ref{prop:finite_dim_projections} we have that equation \eqref{eq:variation_sum_1} can be bounded by
	\begin{align*}
	 C\abs{\A}_{\set{\kappa_{k}}}\left(\kappa_{m}\norm{J}_{A^{2}(U)\to A^{2}(W)}\sup_{z \in U}\abs{f(\cdot ,z)}_{\set{\kappa_{k}}}+ \norm{J-J_{n}}_{A^{2}(U) \to A^{2}(W)} \sup_{x}\norm{f(x,\cdot)}_{A^{2}(U)}\right).
	\end{align*}
	Similarly by Proposition \ref{prop:finite_dim_projections} we have that equation \eqref{eq:variation_sum_3} can be bounded by
	\begin{align*}
	C\abs{\varphi}_{\set{\kappa_{k}}}\left(\kappa_{m}\norm{J}_{A^{2}(U)\to A^{2}(W)}\sup_{z \in U}\abs{f(\cdot ,z)}_{\set{\kappa_{k}}}+ \norm{J-J_{n}}_{A^{2}(U) \to A^{2}(W)} \sup_{x}\norm{f(x,\cdot)}_{A^{2}(U)}\right).
	\end{align*}
	Both of these terms are $O\left(\norm{f}_{\B(U)}\kappa_{m} + e^{-cn^{1/d}}\right)$. Notice that by Proposition \ref{prop:finite_dim_projections} (\ref{prop:finite_dim_projections/finalvarbound}) we have that equation \eqref{eq:variation_sum_2} is
	\[ O\left(\norm{f}_{\B(U)}\min\set{\frac{\kappa_{k+1}}{\kappa_{k}}, \frac{\kappa_{m}+e^{-c_{\A}n^{1/d}}}{\kappa_{k}} }\right). \]
	Hence the result.
\end{proof}

\begin{lemma}\label{lem:uniform_bound}
	There exists a constant $C>0$ such that
	\[ \abs{\L_{\beta}f(x,z)- \L_{\beta}E_{n,m}f(x,z)}\leq C \norm{f}_{\B(U)}(\kappa_{m} + e^{-c_{\A}n^{1/d}})  \]
	for all $f \in \B(U)$, $x \in \S_{T}^{+}$ and $z \in U$.
\end{lemma}
\begin{proof}
	Let $x\in \S_{T}^{+}$ and $z \in U$. Then by Lemma \ref{lem:uniform_norm}
	\begin{align*}
	\abs{\L_{\beta}f(x,z)- \L_{\beta}E_{n,m}f(x,z)}&\leq C \sup_{x \in \S_{T}^{+}} \norm{Jf(x, \cdot) - E_{n,m}f(x,\cdot)}_{A^{2}(W)}\\
	& \leq C \norm{J}_{A^{2}(U)\to A^{2}(W)}\sup_{z \in U}\var_{m}f(\cdot ,z) \\ &\;\;\;\;\;\;\;\;+ C \norm{J-J_{n}}_{A^{2}(U) \to A^{2}(W)} \sup_{x}\norm{f(x,\cdot)}_{A^{2}(U)}\\
	&= \norm{f}_{\B(U)}O(\kappa_{m} + e^{-c_{\A}n^{1/d}}).
	\end{align*}
\end{proof}

\begin{proposition}\label{prop:finitedimapproxbound}
	Let $\beta \in \R$. Then there exists a constant $C>0$ such that
	\[ \norm{\L_{\beta} - \L_{\beta}E_{n,m}}_{\B(U), \op} \leq C \left(\max\set{\sup_{k\geq 1}\min\set{\frac{\kappa_{k+1}}{\kappa_{k}}, \frac{\kappa_{m}+e^{-c_{\A}n^{1/d}}}{\kappa_{k}} } , \kappa_{m}+e^{-c_{\A}n^{1/d}} }\right). \]
\end{proposition}
\begin{proof}
	This follows immediately from Lemma \ref{lem:lipbound} and Lemma \ref{lem:uniform_bound}.
\end{proof}

\begin{lemma}\label{lem:calc}
	For all $p\geq 1$ and $k\geq 1$ we have that
	\[ \frac{(k+1)^{p}-k^{p}}{k^{p-1}} \geq p. \]
\end{lemma}
\begin{proof}
	Notice that 
	\[ \frac{(k+1)^{p}-k^{p}}{k^{p-1}}  \]
	is decreasing and therefore bounded below by
	\[ \lim_{k \to \infty}\frac{(k+1)^{p}-k^{p}}{k^{p-1}} = p. \]
\end{proof}

In order to ensure that $\norm{\L_{\beta} - \L_{\beta}E_{n,m}}_{\B(U), \op}$ will converge to $0$ as $n,m \to \infty$ we need to take some assumptions on $\set{\kappa_{n}}$. If $\kappa_{k}=e^{-ck^{p}}$ for $p\geq 1$ then take $(n,m)=(m^{(p+1)d},m)$. Then 
\[ \kappa_{m}+e^{-c_{\A}n^{1/d}}  = e^{-cm^{p}} + e^{-c_{\A}m^{\floor{p+1}}} \leq C e^{-cm^{p-1}}. \]
for some $C>0$. If $k<m$ then 
\[ \min\set{\frac{\kappa_{k+1}}{\kappa_{k}}, \frac{\kappa_{m}+e^{-c_{\A}m^{\floor{p+1}}}}{\kappa_{k}}} \leq \frac{\kappa_{m}+e^{-c_{\A}m^{p+1}}}{\kappa_{k}}\leq  C e^{-c(m^{p}-k^{p})}. \]
Notice that $m^{p}-k^{p} \geq m^{p}-(m-1)^{p}$ by Lemma \ref{lem:calc} we have that
\[  e^{-c(m^{p}-k^{p})} \leq e^{-pcm^{p-1}}.\]
Thus
\[ \min\set{\frac{\kappa_{k+1}}{\kappa_{k}}, \frac{\kappa_{m}+e^{-c_{\A}m^{\floor{p+1}}}}{\kappa_{k}}} \leq Ce^{-pcm^{p-1}}\] 
for some $C>0$. If $k \geq m-1$ then 
\[ \min\set{\frac{\kappa_{k+1}}{\kappa_{k}}, \frac{\kappa_{m}+e^{-c_{\A}m^{\floor{p+1}}}}{\kappa_{k}}}\leq \frac{\kappa_{k+1}}{\kappa_{k}}= e^{-c((k+1)^{p}-k^{p})} \leq e^{-pck^{p-1}} \leq e^{-pcm^{p-1}}. \]
Therefore
\[ \min\set{\frac{\kappa_{k+1}}{\kappa_{k}}, \frac{\kappa_{m}+e^{-c_{\A}m^{\floor{p+1}}}}{\kappa_{k}}}\leq C e^{-pcm^{p-1}}. \]
for some $C>0$.

\begin{theorem}\label{thm:approxnumbers}
	Let $p \geq 1$ and $q \in \Z^{+}$. If $\kappa_{n}=O(e^{-cn^{p}})$ then for all $\e>0$ there exists a constant $C>0$ such that 
	\[ a_{n}(\L_{\beta}^{q}) \leq C \left(\exp \left[-pqc\left( \log n^{k} \right)^{p-1}\right] \right) \text{ where }k=\frac{1}{h_{\tope}+\e}. \]
\end{theorem}
\begin{proof}
	We follow the method in \cite{MR4147353}. For $q \in \Z^{+}$ define
	\[ E_{n,m}^{(q)} = \L_{\beta}^{q}-(\L_{\beta}- \L_{\beta}E_{n,m})^{q}.  \]
	Notice that expanding $(\L_{\beta}-\L_{\beta}E_{n,m})^{q}$ we obtain 
	\[ E_{n,m}^{(q)}=\sum_{k=0}^{q-1}\L_{\beta}^{q-k}E_{n,m}(\L_{\beta}-\L_{\beta}E_{n,m})^{k}.  \]
	Thus $\ran(E_{n,m}^{(q)})\leq qn \cdot \abs{L_{m}}$. Let $n \geq 1$ and $\e>0$. 
	Notice that
	\begin{align*}
	qm^{\floor{p+1}d}\abs{L_{m}} &=q \frac{m^{\floor{p+1}d}\abs{L_{m}}}{e^{m(h_{\tope}+\e)}}e^{m(h_{\tope}+\e)}
	\end{align*}
	as 
	\[ q\frac{m^{\floor{p+1}d}\abs{L_{m}}}{e^{m(h_{\tope}+\e)}}\xrightarrow{m \to \infty}0 \]
	we have that for sufficiently large $m$
	\begin{align*}
	qm^{\floor{p+1}d}\abs{L_{m}} < e^{m(h_{\tope}+\e)}.
	\end{align*}
	For the remainder of the the proof we will set 
	\[ m = \floor{\frac{\log n}{h_{\tope}+\e}}. \]
	Then for sufficiently large $n$ we have that
	\begin{gather*}
	m  \leq \frac{\log n}{h_{\tope}+\e} \\
	m(h_{\tope}+\e)  \leq  \log n \\
	e^{m(h_{\tope}+\e)} \leq n \\
	qm^{\floor{p+1}d}\abs{L_{m}}  < n.
	\end{gather*}
	Thus for sufficiently large $n$
	\[ \ran(E_{m^{\floor{p+1}},m}^{(q)}) \leq q m^{\floor{p+1}d}\abs{L_{m}} < n.  \]
	Therefore for sufficiently large $n$
	\begin{align*}
	a_{n}(\L_{\beta})&\leq \norm{\L_{\beta} - E_{m^{\floor{p+1}},m}^{(q)}}_{\B, \op}\\
	& = \norm{(\L_{\beta}-\L_{\beta}E_{n,m})^{q}}_{\B. \op}\\
	&\leq \norm{\L_{\beta}-\L_{\beta}E_{n,m}}_{\B. \op}^{q}\\
	&\leq Ce^{-pqc m^{p-1} }\\
	&=C\exp\left[-pqc \left(\floor{\frac{\log n}{h_{\tope}+\e}}\right)^{p-1} \right]\\
	&\leq C' \exp\left[-pqc \left(\frac{\log n}{h_{\tope}+\e} \right)^{p-1}\right].
	\end{align*}
	Hence the result.
\end{proof}

Now using Weyl's inequality we can prove the following from which Theorem \ref{thm:eigenvalues} will follow.

\begin{theorem}\label{thm:eigenvaluestech}
	Let $\S_{T}^{+}$ be the shift of finite type defined by the matrix $T$, $\varphi:\S_{T}^{+} \to \R$, and $\A:\S_{T}^{+} \to GL_{d}(\R)$ be such that $\A(\S_{T}^{+})$ is dominated and
	\[ \var_{n}\varphi, \var_{n}\A = O(e^{-cn^{p}}). \]
	There is a Banach space $\B$ dense in a quotient of $C(\S_{T}^{+}\times \RP^{d-1})$ such that for any $\beta$ $\L_{\beta}$ acts on $\B$ and the following are true.
	\begin{enumerate}
		\item
		If $p>1$ then $\L_{\beta}$ is compact.
		
		\item 
		If $p=2$ then for any $\e>0$
		\[ \abs{\lambda_{n}(\L_{\beta})} = O(n^{-k}) \]
		where $\set{\lambda_{n}(\L_{\beta})}$ is the sequence of eigenvalues of $\L_{\beta}$ listed in descending order of modulus with algebraic multiplicity and
		\[ k=\frac{2c}{h_{\tope}+\e}-\frac{1}{2}. \]
	\end{enumerate}
\end{theorem}

\begin{proof}
		First notice that if $p>1$ then $a_{n}(\L_{\beta})$ converges to $0$ and hence $\L_{\beta}$ is compact. If $p=2$ then by Theorem \ref{thm:weyl} we have that
		\begin{align*}
		\abs{\lambda_{n}(\L_{\beta})}&\leq \left(\prod_{k=1}^{n}\abs{\lambda_{k}}\right)^{1/n}\\
		&\leq n^{1/2}\left(\prod_{k=1}^{n}a_{k}(\L_{\beta})\right)^{1/n}\\
		&\leq n^{1/2}C' n!^{\left(\frac{-2c}{h_{\tope}+\e}\right)\frac{1}{n}}\\
		&\leq C'' n^{\frac{-2c}{h_{\tope}+\e}+\frac{1}{2}}
		\end{align*}
		for some $C''>0$.
\end{proof}

Notice that Theorem \ref{thm:eigenvalues} follows by taking $\e = h_{\tope}$.

\section{The trace formula}\label{sec:traceformula}

Throughout this section we assume that $\A$ satisfies \textbf{H1} and that $\varphi, \A$ satisfy \textbf{H2}$(\set{\kappa_{n}})$ where
\[  \kappa_{n}=e^{-cn^{2}} \text{ and }c > h_{\tope}. \]
We begin by recalling some basic facts about traces and determinants of operators acting on Banach spaces. Define
\[ \mathfrak{L}_{1}^{(a)} = \set{\T \in \mathfrak{L}(\B(U)) : \sum_{n=1}^{\infty}a_{n}(\T)<\infty} \]
$\mathfrak{L}_{1}^{(a)}$ is a two sided ideal in $\L(\B(U))$. This definition sheds light on why we will require that $c>h_{\tope}$ notice that taking $\e = h_{\tope}$ (for convenience) in Theorem \ref{thm:approxnumbers} we have that 
\[ a_{n}(\L_{\beta}) = O(n^{-k}) \]
where
\[ k= \frac{c}{h_{\tope}} > 1. \]
Hence we have that $\L_{\beta} \in \mathfrak{L}^{(a)}_{1}$. We recall some facts about $\mathfrak{L}^{(a)}_{1}$, the following can be found in \cite{MR890520}.

\begin{proposition}\label{prop:tracecontinuity}
	Assume that $\T \in \mathfrak{L}^{(a)}_{1}$ then following are true:
	\begin{enumerate}
		\item 
		Let $\set{\lambda_{n}(\T)}_{n=1}^{\infty}$ be the eigenvalues of $\T$ listed with multiplicity then
		\[ \sum_{n=1}^{\infty}\abs{\lambda_{n}(\T)} < \infty. \]
		
		\item
		There is a unique continuous trace, $\tr$, on $\mathfrak{L}_{1}^{(a)}$ which satisfies the formula
		\[ \tr(\T)=\sum_{n=1}^{\infty}\lambda_{n}(\T). \]
		
		\item
		$\tr$ is continuous in the following sense. If 
		\[ \lim_{m\to \infty}\sum_{k=1}^{\infty}a_{k}(\T_{m}-\T) = 0 \]
		then
		\[ \lim_{m\to \infty}\tr(\T_{m})=\tr(\T). \]
	\end{enumerate}
\end{proposition}

The following lemma is well known in the context of dynamical systems it goes back to Ruelle \cite{MR420720}. One can find a proof in the generality that we require in \cite{morris2018fast}.

\begin{lemma}
	Let $\varphi_{m}(x) = \sum_{\abs{I}=m}\varphi(x_{I})\chi_{[I]}(x)$. and $\A_{m}(x)=\sum_{\abs{I}=m}\A(x_{I})\chi_{[I]}(x)$. Then 
	\[ \tr (\L_{\varphi_{m}, \A_{m} ,\beta}^{n}) = \sum_{x \in \per_{n}(\sigma)} e^{S_{n}\varphi_{m}(x)}\frac{\lambda_{1}(\A^{(n)}_{m}(x))^{d-1}\rho(\A^{(n)}_{m}(x))^{\beta}}{p_{\A^{(n)}_{m}(x)}'(\lambda_{1}(\A^{(n)}_{m}(x)))} \]
	where $p_{A}'(\cdot)$ is the derivative of the characteristic polynomial of $A$.
\end{lemma}

\begin{lemma}\label{lem:convergence}
	For all $n \geq 1$ $\L_{\varphi_{m},\A_{m},\beta}^{n}$ converges to $\L_{\beta}^{n}$ in the $\norm{\cdot}_{\B(U)}$ norm.
\end{lemma}

The proof of Lemma \ref{lem:convergence} is straight forward but long so we postpone it to the appendix.

\begin{lemma}
	For all $n \geq 1$
	\[ \lim_{m\to \infty}\sum_{k=1}^{\infty}a_{k}(\L_{\varphi_{m},\A_{m},\beta}^{n}- \L_{\beta}^{n}) = 0. \]
\end{lemma}
\begin{proof}
	The proof is the same as \cite[Corollary 5.4]{MR4147353} we provide it here for the sake of completeness. We will prove the result for $n=1$ the proof for $n \geq 1$ is similar. We will write $\L_{m,\beta}$ for $\L_{\varphi_{m}, \A_{m},\beta}$. Let $\e>0$. Recall that there exists and $N$ and a constant $C>0$ such that
	\[ a_{k}(\L_{\beta}) \leq C\exp\left[-2c \left(\frac{\log k}{2h_{\tope}} \right)\right] = Ck^{\frac{-c}{h_{\tope}}} \]
	As $c > h_{\tope}$ we can take $n_{0}$ such that 
	\[ \sum_{k \geq n_{0}} a_{k}(\L_{\beta}) < \e. \]
	Notice that as $\var_{n}\varphi_{m}, \var_{n}\A_{m} = O(e^{-cn^{2}})$ for the same $c>0$ we have that
	\[ \sum_{k \geq n_{0}} a_{k}(\L_{m,\beta}) < \e. \]
	By Lemma \ref{lem:convergence} we can take $m_{0}$ large enough such that 
	\[ n_{0}\norm{\L_{m,\beta}-\L_{\beta}}_{\B(U)} < \e  \]
	for all $m \geq m_{0}$. Now we write
	\[ \sum_{k=1}^{\infty}a_{k}(\L_{m,\beta} - \L_{\beta}) = \sum_{k=1}^{2n_{0}}a_{k}(\L_{m,\beta} - \L_{\beta}) + \sum_{k = 2n_{0}+1}^{\infty}a_{k}(\L_{m,\beta} - \L_{\beta}) \]
	and bound each term individually. Notice
	\[ \sum_{k=1}^{2n_{0}}a_{k}(\L_{m,\beta} - \L_{\beta})\leq 2n_{0}\norm{\L_{m,\beta}-\L_{\beta}}_{\B(U)} < 2 \e.  \]
	Furthermore as $a_{2k}(\L_{m,\beta} - \L_{\beta})\leq a_{2k-1}(\L_{m,\beta} - \L_{\beta}) \leq a_{k}(\L_{m,\beta}) + a_{k}(\L_{\beta})$ we have that
	\begin{align*}
	\sum_{k = 2n_{0}+1}^{\infty}a_{k}(\L_{m,\beta} - \L_{\beta})& = \sum_{l=n_{0}}^{\infty}a_{2k}(\L_{m,\beta} - \L_{\beta}) + \sum_{l=n_{0}+1}^{\infty}a_{2k-1}(\L_{m,\beta} - \L_{\beta}) \\
	&\leq 2\left[\sum_{l=n_{0}}^{\infty}a_{k}(\L_{m,\beta}) + \sum_{l=n_{0}}^{\infty}a_{k}(\L_{\beta})\right] < 4\e.
	\end{align*}
	Therefore
	\[ \sum_{k=1}^{\infty}a_{k}(\L_{m,\beta} - \L_{\beta}) < 6\e. \]
	Hence the result.
\end{proof}

Combining Proposition \ref{prop:tracecontinuity} and the previous lemma we have the following.

\begin{proposition}\label{prop:traceformula}
	If $\L_{\beta}^{n} \in \mathfrak{L}^{(a)}$ then 
	\[ \tr(\L_{\beta}^{n})= \sum_{x \in \per_{n}(\sigma)} e^{S_{n}\varphi(x)}\frac{\lambda_{1}(\A^{(n)}(x))^{d-1}\rho(\A^{(n)}(x))^{\beta}}{p_{\A^{(n)}(x)}'(\lambda_{1}(\A^{(n)}(x)))}. \]
\end{proposition}

\section{The  Determinant and a Formula for the Lyapunov Exponent}\label{sec:determinant}

In this section we will make the connection between the operators $\L_{\beta}$ and the Lyapunov exponent for the cocycle generated by $\A$ over the Gibbs state $\mu_{\varphi}$ and produce a formula for the Lyapunov exponent. The general method is the same as \cite{MR2651384}, \cite{MR4014663} so we will describe it somewhat briefly. The connection arises from the following fact, if $\rho(\L_{0})=1$
\[ \frac{d \rho(\L_{\beta})}{d \beta} \Big|_{\beta =0} = \gamma_{1}(\A, \mu_{\varphi}). \]
This fact is well known in our setup it can be deduced in the same way as \cite{park2020transfer}. Notice that in order to ensure that $\rho(\L_{\beta})$ is differentiable we need that $T$ is irreducible. In our case as the operators $\L_{\beta}$ are in $\mathfrak{L}^{(a)}$ there is an associated determinant $\delta(\zeta , \L_{\beta})$ for which $\delta(-\rho(\L_{\beta})^{-1}, \L_{\beta}) = 0$. Expanding $\delta(-\rho(\L_{\beta})^{-1}, \L_{\beta})$ as a Taylor series we have that
\[ 0 = \delta(-\rho(\L_{\beta})^{-1}, \L_{\beta}) = \sum_{n=0}^{\infty}(-1)^{n}\alpha_{n}(\L_{\beta})\rho(\L_{\beta})^{-n}. \]
Differentiating with respect to $\beta$ we find that
\[ \frac{d \rho(\L_{\beta})}{d \beta} = \frac{\sum_{n=0}^{\infty}(-1)^{n}\frac{d \alpha_{n}(\L_{\beta})}{d\beta}\rho(\L_{\beta})^{-n}}{\sum_{n=0}^{\infty}(-1)^{n} n \alpha_{n}(\L_{\beta}) \rho(\L_{\beta})^{-n-1}}  \]
If we assume that $\varphi$ is a $g$-function (so that $\rho(\L_{0}) =1$) and denote $\frac{d \alpha_{n}(\L_{\beta})}{d\beta}\Big|_{\beta = 0}$ by $\alpha_{n}'(0)$ and $\alpha_{n}(\L_{\beta})$ by $\alpha_{n}(\beta)$ we find that
\[ \gamma_{1}(\A,\mu_{\varphi}) =  \frac{\sum_{n=0}^{\infty}(-1)^{n}\alpha_{n}'(0)}{\sum_{n=0}^{\infty}(-1)^{n} n \alpha_{n}(0)}.  \]
Hence to obtain a method for approximating $\gamma_{1}(\A,\mu_{\varphi})$ we need a way to compute $\alpha_{n}(0)$ and $\alpha_{n}'(0)$ and also rates at which they converge to $0$.

First lets recall come facts about determinants, the following results can be found in \cite{MR890520}. The fact that $\mathfrak{L}_{1}^{(a)}$ admits a spectral trace gives us access to a suitable determinant theory. Recall the \emph{Fredholm resolvent} of $T$ is
\[ F(\zeta , T)=T(I+\zeta T)^{-1}. \]
Given a continuous determinant $\delta$ defined on an ideal $\mathfrak{U}$ the associated Fredholm denominator
\[ \delta(\zeta , T) = \delta(I+\zeta T) \]
is an entire function whose zeros $\zeta_{0}$ are related to the eigenvalues of $T$ by $\zeta_{0}=1/\lambda_{0}$. There exists an entire operator valued function, $D(\zeta , T)$ such that 
\[ F(\zeta , T)=\frac{D(\zeta , T)}{\delta(\zeta , T)}. \]

\begin{proposition}
	Let $\tau$ be the spectral trace defined on a quasi-Banach operator ideal $\mathfrak{U}$. If $T \in \mathfrak{U}(E)$ then the associated Fredholm denominator is given by the formula
	\[ \delta(\zeta , T) = \prod_{n=1}^{\infty}(1+\zeta \lambda_{n}(T)) \]
	for all $\zeta \in \C$.
\end{proposition}

\begin{theorem}
	Let $\tau$ be the trace on $\mathfrak{L}_{1}^{(a)}$. If $T \in \mathfrak{L}_{1}^{(a)}$ then the coefficients of the associated Fredholm denominator 
	\[ \delta(\zeta , T)=\sum_{n=0}^{\infty}\alpha_{n}(T)\zeta^{n} \]
	and of the associated Fredholm numerator
	\[ D(\zeta , T)=\sum_{n=0}^{\infty}A_{n}\zeta^{n} \]
	are determined in the following way:
	\[ \alpha_{0}(T) = 1 \text{ and }A_{0}(T)=T. \]
	For $n \geq 1$ we have that 
	\begin{eqnarray}
	\alpha_{n}(T)=\frac{1}{n}\tau(A_{n-1}(T)) \\
	A_{n}(T) = \sum_{h=0}^{n}(-1)^{n-h}\alpha_{h}(T)T^{n-h+1}
	\end{eqnarray}
	In addition we have 
	\begin{eqnarray}
	\alpha_{n}(T) = \frac{1}{n!}\det\pmat{\tau(T)&n-1& & &0 \\ 
		\tau(T^{2})&\tau(T)&n-2& & \\
		\vdots & \vdots & &\ddots & \\
		\tau(T^{n-1})& \tau(T^{n-2}) & \cdots & \tau(T)& 1 \\
		\tau(T^{n})&\tau(T^{n-1})& \cdots &\tau(T^{2})&\tau(T) } \\
	A_{n}(T) = \frac{1}{n!}\det\pmat{\tau(T)&n& & &0 \\ 
		\tau(T^{2})&\tau(T)&n-1& & \\
		\vdots & \vdots & &\ddots & \\
		\tau(T^{n})& \tau(T^{n-1}) & \cdots & \tau(T)& 1 \\
		T^{n+1}&T^{n}& \cdots &T^{2}& T }.
	\end{eqnarray}
\end{theorem}

By Proposition \ref{prop:traceformula} we have a formula for $\tr(\L_{\beta}^{n})$ in terms of the $\varphi$ and the eigenvalues for $\A^{(n)}(p)$ where $p$ is a periodic point of period $n$. Hence using the proceeding theorem we may obtain a formula for $\alpha_{n}(0)$ and $\alpha_{n}'(0)$ in terms of $\varphi$ and the eigenvalues for $\A^{(n)}(p)$ where $p$ is a periodic point of period $\leq n$.

Next we turn our attention to determining the rate at which $\alpha_{n}(0)$ and $\alpha_{n}'(0)$ converge to $0$. Notice that because $\beta \mapsto \alpha_{n}(\beta)$ is analytic by Cauchy's integral formula if we can bound $\alpha_{n}(\beta)$ for all $\beta$ small we will have the same bound for $\alpha_{n}'(0)$. Before diving into the details of the proof we state the theorem.

\begin{theorem}\label{thm:taylorbound}
	Suppose that 
	\[  \kappa_{n}=O(e^{-cn^{2}}) \text{ where }c > h_{\tope}. \]
	For any $\beta$ the Taylor coefficients of the determinant for $\L_{\beta}$ satisfy
	\[ \alpha_{n}(\beta) = O(n^{-nk}). \]
	Where
	\[ k = \frac{2c - h_{\tope}}{4 h_{\tope}}. \]
\end{theorem}

We comment that the implied constant in Theorem \ref{thm:taylorbound} can be taken uniform for $\beta \leq 1$. This is because the implied constants used in the previous section depend continuously on $\beta$. Hence we may also conclude that
\[ \alpha_{n}'(0) = O(n^{-nk}) \]
for the same $k$. Once we establish Theorem \ref{thm:taylorbound} we will have the following which is the main result of the paper.

\begin{theorem}\label{thm:mainconcrete}
	Let $\S_{T}^{+}$ be the shift of finite type defined by an irreducible matrix $T$, $g:\S_{T}^{+} \to \R$ be a $g$-function, $\A:\S_{T}^{+} \to GL_{d}(\R)$ be such that $\A(\S_{T}^{+})$ is dominated and
	\[ \var_{n}\log g, \var_{n}\A = O(e^{-cn^{2}})\text{ where }c > h_{\tope}. \]
	Let $\alpha_{i}(0)$ and $\alpha_{i}'(0)$ be as above and define
	\[ \gamma_{1}^{(n)}(\A,\mu_{g}) = \frac{\sum_{i=0}^{n}(-1)^{i}\alpha_{i}'(0)}{\sum_{i=0}^{n} (-1)^{i}i \alpha_{i}(0)} \]
	Then $\gamma_{1}^{(n)}(\A,\mu_{g})$ is computable using the values of $g$ and $\A$ at periodic points of period $\leq n$ and
	\[ \abs{\gamma_{1}^{(n)}(\A,\mu_{g}) - \gamma_{1}(\A,\mu_{g})} = O(n^{-nk}) \]
	where
	\[ k = \frac{2c - h_{\tope}}{4 h_{\tope}}. \]
\end{theorem}

To prove the previous theorem we will need some facts about analytic functions. The following results can be found in \cite{alma9917579304202441}.

\begin{notation}
	Let $f$ be an entire function. Define
	\[ M(r)=\max_{\abs{z}=r}\abs{f(z)} \]
	Define $n(t)$ be the number number of zeros of $f$ in $\abs{z}\leq t$ counted with multiplicity. Define
	\[ N(r)=\int_{0}^{r}t^{-1}n(t) dt. \]
\end{notation}

\begin{definition}
	The entire function $f$ is of order $\rho$ if 
	\[ \rho = \limsup_{r \to \infty}\frac{\log \log M(r)}{\log r}. \]
\end{definition}

\begin{theorem}
	The entire function 
	\[ f(z)=\sum_{n=0}^{\infty} a_{n}z^{n} \]
	is of finite order if and only if 
	\[ \mu = \limsup_{n \to \infty}\frac{n \log n}{\log(1/\abs{a_{n}})} \]
	is finite. In this case the order of $f$ is equal to $\mu$.
\end{theorem}

\begin{lemma}\label{lem:Mbound}
	If $f(z)$ is an entire function of genus $0$ with $f(0)=1$ we have 
	\[ \log M(r) \leq N(r) + Q(r) \]
	where 
	\[ Q(r) = r \int_{r}^{\infty}t^{-2}n(t)dt. \]
\end{lemma}

\begin{proof}[proof of Thereom \ref{thm:taylorbound}]
	By Theorem \ref{thm:eigenvalues} we have that $\abs{\lambda_{n}}\leq b_{n}=Cn^{-k}$ where $k>1$ is as in Theorem \ref{thm:eigenvalues} counted with multiplicity. To see that we can take $k>1$ notice that we can set $\e = h_{\tope}/3$ then 
	\[ \frac{2c}{(4/3)h_{\tope}}-\frac{1}{2} > \frac{6}{4}-\frac{1}{2} = 1.  \]
	Thus $b_{n}\in \ell^{1}$. Notice
	\[ \abs{\delta(\zeta , T)}=\prod_{n=1}^{\infty}\abs{(1+\zeta \lambda_{n}(T))} \leq \prod_{n=1}^{\infty}1+\abs{\zeta}\abs{\lambda_{n}(T)} \leq \prod_{n=1}^{\infty}1+\abs{\zeta}b_{n}.  \]
	Thus we have that for any $z \in \C$ with $\abs{z}=r$
	\[ \abs{\delta(z,t)}\leq \prod_{n=1}^{\infty}1+rb_{n}. \]
	Define
	\[ f(\zeta) = \prod_{n=1}^{\infty}1+\zeta b_{n} \]
	as $b_{n}\in \ell^{1}$ we have that $f$ is well defined and is an entire function moreover for any $r$
	\[  M_{\delta(\cdot , t)}(r) \leq M_{f}(r).  \]
	So to bound $M_{\delta(\cdot , T)}$ it suffices to bound $M_{f}(r)$. Notice that $f$ is a genus $0$ function. Thus by Lemma \ref{lem:Mbound} we have that 
	\[ \log M_{\delta(\cdot , T)}(r)\leq N_{f}(r) + Q_{f}(r). \]
	Thus to bound $M_{f}(r)$ we must bound $n_{f}(t)$. Notice that the zeros of $f(z)$ are $-1/b_{n}= -b_{n}^{-1}$ and that $\abs{-b_{n}^{-1}} = \abs{b_{n}}^{-1}=C^{-1}n^{k}$
	\begin{gather*}
	Cn^{k} = t\\
	\log C + k \log n =  \log t \\
	\log n  = k^{-1}\log(t/C) \\
	n = \exp \left[  k^{-1}\log(t/C) \right] = \exp \left[  \log(t/C)^{k^{-1}} \right] = (t/C)^{k^{-1}}
	\end{gather*}
	thus $n(t) \leq C't^{1/k}$
	\begin{align*}
	N(r) \leq C'\int_{0}^{r}t^{-1+1/k}dt = kC'r^{1/k}
	\end{align*}
	
	\begin{align*}
	Q(r)\leq C'r \int_{r}^{\infty} t^{-2+1/k}dt = \frac{k}{1-k}C'r (r^{-1+1/k}) = \frac{k}{1-k}C'r^{1/k}
	\end{align*}
	
	Thus
	\[ \log M_{\delta(\cdot , T)}(r) \leq C'' r^{1/k}. \]
	Therefore
	\begin{align*}
	\rho &= \limsup_{r \to \infty}\frac{\log \log M_{\delta(\cdot , T)}}{\log r} \\
	&\leq \limsup_{r \to\infty}\frac{\log C''r^{1/k}}{\log r} \\
	& = \limsup_{r \to \infty}\frac{\log C''}{\log r} + \limsup_{r \to \infty}\frac{1/k\log r}{\log r}\\
	&=k^{-1}
	\end{align*}
	
	Thus
	\[ \limsup_{n \to \infty} \frac{n \log n}{-\log \abs{\alpha_{n}(T)}} \leq k^{-1} \]
	
	For all large $n$ we have that $R=k^{-1}+\e$
	\begin{gather*}
	\frac{n \log n}{-\log \abs{\alpha_{n}(T)}} \leq R\\
	\log n^{n} \leq -R \log \abs{\alpha_{n}(T)} \\
	\log n^{-n/R} \geq \log \abs{\alpha_{n}(T)} \\
	n^{-n/R} \geq \abs{\alpha_{n}(T)}
	\end{gather*}
	Thus
	\[ \abs{\alpha_{n}(T)} \leq n^{-n/(k^{-1}+\e)} \]
	Therefore there exists a constant such that take $\e = 1/k$
	\[ \abs{\alpha_{n}(T)}\leq C n^{-n/(k^{-1}+\e)} = C n^{\frac{-nk}{1+\e k}} = Cn^{\frac{-nk}{2}} \]
	for all $n \geq 1$.
\end{proof}

\section{An Example}\label{sec:example}

In this section we will give an example which demonstrates the effectiveness of using these approximation in practice. The code used to perform these computations can be found at https://github.com/mpiraino/fastApproxLE. Consider the following example:
\[ \varphi = 1/2 \text{ and }\A(x)=\begin{cases}
\begin{bmatrix}2 & 1+2^{-m^{3}}\\ 1+3^{-m^{3}} & 1\end{bmatrix}& x = 0^{m}1\\ \hspace{4em}
\begin{bmatrix}2 & 1\\ 1 & 1\end{bmatrix} & x =0^{\infty}
\end{cases} \]
That is $\mu_{\varphi}$ is the $(1/2, 1/2)$ Bernoulli measure and $\A$ is some cocycles which is not locally constant. Notice that $\var_{n}\A = O(2^{-m^{3}})$ and $\A(x)$ is positive for any $x$ so that it satisfies the assumptions of Theorem \ref{thm:mainconcrete}.

The naive method for estimating $\gamma_{1}(\A , \mu_{\varphi})$ would be to use the approximation
\begin{equation}\label{eq:basicapprox}
\frac{1}{n} \sum_{\abs{I}=n}e^{S_{n}\varphi(x_{I})}\log\norm{\A^{(n)}(x_{I})} 
\end{equation}
where $x_{I}$ is some arbitrary point in $[I]$. The logic being that
\[ \sum_{\abs{I}=n}e^{S_{n}\varphi(x_{I})}\log\norm{\A^{(n)}(x_{I})} \approx \int \log\norm{\A^{(n)}(x)}d\mu_{\varphi}. \]
In practice the quality of this approximation is dependent on the choice of norm, so we have used a few common norms. The following table summarizes the results of our computations. We have highlighted the number of digits for which the approximation appears to be accurate.

\begin{center}
	\begin{tabular}{l | c | c | c | c}
		max period	&   Theorem \ref{thm:mainconcrete} & Equation \eqref{eq:basicapprox} $\norm{\cdot}_{2}$ &   Equation \eqref{eq:basicapprox} $\norm{\cdot}_{1}$ &   Equation \eqref{eq:basicapprox} $\norm{\cdot}_{\infty}$ \\ \hline
		1          &  \colorbox{Mycolor1}{1.}09308925851915 &  \colorbox{Mycolor1}{1.1}2771487662921 &  \colorbox{Mycolor1}{1.}77767403074471 &  0.693147180559945 \\
		2          &   \colorbox{Mycolor1}{1.113}99675194920 & \colorbox{Mycolor1}{1.11}501540995010 &  \colorbox{Mycolor1}{1.}44557108726526 &  0.909049799071256 \\
		3          &  \colorbox{Mycolor1}{1.11336}708955451 & \colorbox{Mycolor1}{1.11}435697806841 &  \colorbox{Mycolor1}{1.}33483695545302 &  0.977223409851798 \\
		4          &  \colorbox{Mycolor1}{1.11336692026}619 &  \colorbox{Mycolor1}{1.11}410727056611 &  \colorbox{Mycolor1}{1.}27946945445769 &   \colorbox{Mycolor1}{1.}01126018442876 \\
		5          &  \colorbox{Mycolor1}{1.11336692026723} &  \colorbox{Mycolor1}{1.113}95915553119 &  \colorbox{Mycolor1}{1.}24624894772951 &   \colorbox{Mycolor1}{1.}03168154423466 \\
		6          &  \colorbox{Mycolor1}{1.11336692026723} &  \colorbox{Mycolor1}{1.113}86044871997 &  \colorbox{Mycolor1}{1.}22410194315408 &   \colorbox{Mycolor1}{1.}04529577375891 \\
		7          &  \colorbox{Mycolor1}{1.11336692026723} &  \colorbox{Mycolor1}{1.113}78994463571 &  \colorbox{Mycolor1}{1.}20828265417027 &    \colorbox{Mycolor1}{1.}05502022326290 \\
		8          &  \colorbox{Mycolor1}{1.11336692026723} & \colorbox{Mycolor1}{1.113}73706658924 &  \colorbox{Mycolor1}{1.1}9641818743239 &   \colorbox{Mycolor1}{1.}06231356038848 \\
		9          &  \colorbox{Mycolor1}{1.11336692026723} & \colorbox{Mycolor1}{1.113}69593922012 &  \colorbox{Mycolor1}{1.1}8719026885848 &   \colorbox{Mycolor1}{1.}06798615593057 \\
		10         &  \colorbox{Mycolor1}{1.11336692026723} & \colorbox{Mycolor1}{1.113}66303732483 &  \colorbox{Mycolor1}{1.1}7980793399935 &   \colorbox{Mycolor1}{1.}07252423236423 \\
		11         &  \colorbox{Mycolor1}{1.11336692026723} & \colorbox{Mycolor1}{1.113}63611759232 &  \colorbox{Mycolor1}{1.1}7376784184189 &   \colorbox{Mycolor1}{1.}07623720399178 \\
		12         &  \colorbox{Mycolor1}{1.11336692026723} & \colorbox{Mycolor1}{1.113}61368448190 &  \colorbox{Mycolor1}{1.1}6873443171067 &   \colorbox{Mycolor1}{1.}07933134701473
	\end{tabular}
\end{center}

We can see that the approximations from Theorem \ref{thm:mainconcrete} appear to be accurate to about 15 digits using periodic points of length $\leq 5$ whereas even using cylinders of length 12 the approximation in Equation \eqref{eq:basicapprox} only appears to be accurate to about 4 digits.

\section{Appendix}\label{sec:appendix}

We collect a number of technical results in this appendix which we use throughout the paper. We also include some proofs which we have deferred in our exposition.

\begin{lemma}\label{lem:analyticinequalities}
	Let $U \subseteq \C^{k}$ be open and $f:U \to \C$ analytic.
	\begin{enumerate}
		\item 
		Suppose that $B(z, \e )\subseteq U$ then 
		\[ \abs{f(z)}\leq \sqrt{\frac{k!}{\pi^{k}\e^{k}}}\norm{f}_{A^{2}(U)}. \]
		
		\item 
		Suppose that $K \subseteq U$ is compact and convex. There exists a constant $C_{K}$ depending only on $K$ such that
		\[ \abs{f(z_{1})-f(z_{2})}\leq C_{K} \norm{z_{1}-z_{2}}\norm{f}_{A^{2}(U)} \]
		for all $z_{1},z_{2}\in K$.
	\end{enumerate}
\end{lemma}
\begin{proof}
	\begin{enumerate}
		\item 
		This result is well known (see for instance \cite{morris2018fast}). Notice
		\begin{align*}
		\abs{f(z_{0})}^{2}&=\abs{\frac{1}{V(B(z_{0},\e))}\int_{B(z_{0},\e)}f(z)^{2}dV(z)}\\
		&\leq \frac{1}{V(B(z_{0},\e))}\int_{B(z_{0},\e)}\abs{f(z)}^{2}dV(z)\\
		&\leq \frac{1}{V(B(z_{0},\e))}\int_{U}\abs{f(z)}^{2}dV(z)\\
		&= \frac{k!}{\pi^{k} \e^{k}}\norm{f}_{A^{2}(U)}^{2}.
		\end{align*}
		
		\item 
		Take $\e$ such that $B(z,\e)\subseteq U$ for all $z \in K$. Let $z_{1},z_{2}\in K$ and set $z_{t} = tz_{1}+(\norm{z_{1}-z_{2}}-t)z_{2}$ and define the function $\ol{f}(t)=f(z_{t})$. We claim that $\ol{f}$ is analytic on the set $\bigcup_{0\leq t \leq \norm{z_{1}-z_{2}}}B(t,\e/4S) \subseteq \C$ were $S= \sup_{z \in K}\norm{z}$. To see this notice that for any $w \in \bigcup_{0\leq t \leq \norm{z_{1}-z_{2}}}B(t,\e/4S)$ there exists a $t \in [0, \norm{z_{1}-z_{2}}]$ such that
		\[ \norm{z_{t}-z_{w}}\leq 2 S \abs{t-w} < \e/2. \]
		Thus $z_{w} \in B(z_{t}, \e/2) \subset U$. Thus $\ol{f}$ is a composition of $t \mapsto tz_{1}+(\norm{z_{1}-z_{2}}-t)z_{2}$ and $f$ both of which are analytic. 
		
		Let $\gamma$ be the boundary of the set $\bigcup_{0\leq t \leq \norm{z_{1}-z_{2}}}B(t,\e/4S)$. The $\gamma$ is a piecewise smooth curve and by the Cauchy integral formula we have that
		\begin{align*}
		f(z_{1})-f(z_{2})= \ol{f}(\norm{z_{1}-z_{2}})-\ol{f}(0) &=\frac{\norm{z_{1}-z_{2}}}{2 \pi i}\int_{\gamma}\frac{\ol{f}(w)}{(w-\norm{z_{1}-z_{2}})w}dw
		\end{align*}
		Notice that for any $w \in \gamma$ we have that $\ol{f}(w)=f(z_{w})$ and that $z_{w} \in \ol{B(z_{t},\e/2)}\subseteq B(z_{t}, \e) \subseteq U$ for some $z_{t}\in K$ thus $\dist(z_{w}, \partial U)\geq \e/2$
		\[ \abs{f(z)}\leq \sqrt{\frac{k!2^{k}}{\pi^{k} \e^{k}}}\norm{f}_{A^{2}(U)}^{2}. \]
		Notice that for any $w \in \gamma$ we have that
		\[ \abs{w} \geq \e/4S \text{ and }\abs{w-\norm{z_{1}-z_{2}}} \geq \e/4S \]
		hence
		\[ \abs{(w-\norm{z_{1}-z_{2}})w} \geq \e^{2}/16S. \]
		We can compute that the length of $\gamma$ is
		\[ 2\norm{z_{1}-z_{2}}+ \frac{\pi \e}{16S} \leq 2 \diam(K)+\pi \e. \]
		Thus
		\begin{align*}
		&\abs{\int_{\gamma}\frac{f(w)}{(w-\norm{z_{1}-z_{2}})w}dw}\\
		&\leq \left(2\norm{z_{1}-z_{2}}+ \frac{\pi \e}{16S}\right) \sup\set{\abs{\frac{f(w)}{(w-\norm{z_{1}-z_{2}})w}}:w \in \gamma}\\
		&\leq \norm{f}_{A^{2}(U)}\frac{16S(2\norm{z_{1}-z_{2}}+ \frac{\pi \e}{16S})}{\e^{2}}\sqrt{\frac{k!2^{k}}{\pi^{k} \e^{k}}}.
		\end{align*}
		Hence the result.
	\end{enumerate}
\end{proof}

\begin{lemma}\label{lem:techbounds}
	Suppose that $\A$ satisfies \textbf{(H1)} and let $\ol{\A(x)}$ be as in Theorem \ref{thm:conesmorris}.
	\begin{enumerate}
		\item 
		There is a constant $C_{\A}$ such that
		\[ \norm{\ol{\A(x)}-\ol{\A(y)}} \leq C_{\A} \var_{k}\A \]
		
		\item 
		There is a constant $C_{\A}$ such that
		\[ \norm{ \ol{\A_{m}(x)} - \ol{\A(x)} } \leq C_{\A} \var_{m}\A  \]
		
		\item
		There is a constant $C_{\A, \beta}$ such that 
		\[ \abs{\norm{\A(x)\frac{z}{\norm{z}}}^{\beta} - \norm{\A(y)\frac{z}{\norm{z}}}^{\beta}} \leq C_{\A, \beta} \var_{k}\A \]
		
		\item
		There is a constant $C_{\A, \beta}$ such that 
		\[ \abs{\norm{\A(x)\frac{z}{\norm{z}}}^{\beta}-\norm{\A_{m}(x)\frac{z}{\norm{z}}}^{\beta}} \leq C_{\A, \beta}\var_{m}\A \]

		\item 
		For any $f \in \B(U, \kappa_{n})$
		\[ \abs{f(iy,\ol{\A(iy)}z) - f(ix,\ol{\A(ix)}z)}\leq \norm{f}_{\B(U)}(C_{K}+ 1)\kappa_{k+1}. \]
	\end{enumerate}
\end{lemma}

\begin{proof}
	\begin{enumerate}
		\item
		
		Notice that
		\begin{align*}
		&\norm{\inn{\A(iy)z,w}^{-1}\A(iy)z- \inn{\A(ix)z,w}^{-1}\A(ix)z }\\
		&\leq \norm{\inn{\A(iy)z,w}^{-1}\A(iy)- \inn{\A(ix)z,w}^{-1}\A(ix)}\\
		&=\norm{\frac{\A(iy)\inn{\A(ix)z,w} -\A(ix)\inn{\A(iy)z,w}}{\inn{\A(ix)z,w}\inn{\A(iy)z,w}}}\\
		&=\norm{\frac{\A(iy)\inn{\A(ix)z,w} -\A(iy)\inn{\A(iy)z,w}+\A(iy)\inn{\A(iy)z,w}-\A(ix)\inn{\A(iy)z,w}}{\inn{\A(ix)z,w}\inn{\A(iy)z,w}}}\\
		&\leq \inn{\A(iy)z,w}^{-1}\inn{\A(ix)z,w}^{-1} (\inn{(\A(ix)-\A(iy))z,w}\norm{\A(iy)}\\
		&\;\;\;\;\;\;\;\;\;\;+\inn{\A(iy)z,w}\norm{\A(iy)-\A(ix)})\\
		&\leq C\left(\frac{\norm{\A(iy)}}{\sigma_{d}(\A(iy))\sigma_{d}(\A(ix))}+\sigma_{d}(\A(ix))^{-1} \right)\norm{\A(ix)-\A(iy)}\\
		&\leq C\left(\frac{\sup_{x}\norm{\A(x)}}{(\inf_{x}\sigma_{d}(\A(x)))^{2}} + (\inf_{x}\sigma_{d}(\A(x)))^{-1} \right)\var_{k+1}\A.
		\end{align*}
		
		\item 
		Follows from (1).
		
		\item 
		\begin{align*}
		&\abs{\norm{\A(x)\frac{z}{\norm{z}}}^{\beta} - \norm{\A(y)\frac{z}{\norm{z}}}^{\beta}}\\
		&\leq \beta \max\set{(\inf_{x}\sigma_{d}(\A(x)))^{\beta -1}, (\sup_{x}\norm{\A(x)})^{\beta -1}}	\abs{\norm{\A(x)\frac{z}{\norm{z}}} - \norm{\A(y)\frac{z}{\norm{z}}}}\\
		&\leq \beta \max\set{(\inf_{x}\sigma_{d}(\A(x)))^{\beta -1}, (\sup_{x}\norm{\A(x)})^{\beta -1}}\norm{\A(x)-\A(y)}\\
		&\leq \beta \max\set{(\inf_{x}\sigma_{d}(\A(x)))^{\beta -1}, (\sup_{x}\norm{\A(x)})^{\beta -1}} \var_{k}\A.
		\end{align*}
		
		\item 
		Follows from (3).
		
		\item 
		Notice
		\begin{align*}
		&\abs{f(iy,\ol{\A(iy)}z) - f(ix,\ol{\A(ix)}z)} \\
		&\leq \abs{f(iy,\ol{\A(iy)}z) - f(iy,\ol{\A(ix)}z)} +\abs{f(iy,\ol{\A(ix)}z) - f(ix,\ol{\A(ix)}z)} \\
		&\leq C_{K}\norm{f(iy, \cdot)}_{A^{2}(U)}\norm{\ol{\A(ix)}z-\ol{\A(iy)}z}+ \var_{k+1}f(\cdot , \ol{\A(ix)}z)\\
		&\leq \norm{f}_{\B(U)}(C_{K}+ 1)\kappa_{k+1}.
		\end{align*}
	\end{enumerate}
\end{proof}

\begin{proof}[Proof of Lemma \ref{lem:convergence}]
	First notice that it suffices to prove the result for $n=1$. We will write $\L_{m,\beta} = \L_{\varphi_{m},\A_{m},\beta}$. Let $f \in \B(U)$. Let $x,y \in \S_{T}^{+}$ with $x_{0}=y_{0}$ and $z \in U$ then
	\begin{align*}
		&\abs{(\L_{m,\beta}-\L_{\beta})f(x,z) -(\L_{m,\beta}-\L_{\beta})f(y,z)}\\
		&= \abs{\L_{m,\beta}f(x,z)-\L_{\beta}f(x,z) -\L_{m,\beta}f(y,z)+\L_{\beta}f(y,z)} \\
		&\leq \sum_{i:ix \in \S_{T}^{+}} \Big| e^{\varphi_{m}(ix)}\norm{\A_{m}(ix)\frac{z}{\norm{z}}}^{\beta}f(ix, \ol{\A_{m}(ix)}z) - e^{\varphi(ix)}\norm{\A(ix)\frac{z}{\norm{z}}}^{\beta}f(ix, \ol{\A(ix)}z)  \\
		& \hspace{7em} - e^{\varphi_{m}(iy)}\norm{\A_{m}(iy)\frac{z}{\norm{z}}}^{\beta}f(iy, \ol{\A_{m}(iy)}z) + e^{\varphi(iy)}\norm{\A(iy)\frac{z}{\norm{z}}}^{\beta}f(iy, \ol{\A(iy)}z) \Big|
	\end{align*}
	for each $i$ we have that
	\begin{align*}
		&\Big| e^{\varphi_{m}(ix)}\norm{\A_{m}(ix)\frac{z}{\norm{z}}}^{\beta}f(ix, \ol{\A_{m}(ix)}z) - e^{\varphi(ix)}\norm{\A(ix)\frac{z}{\norm{z}}}^{\beta}f(ix, \ol{\A(ix)}z)  \\
		& \hspace{7em} - e^{\varphi_{m}(iy)}\norm{\A_{m}(iy)\frac{z}{\norm{z}}}^{\beta}f(iy, \ol{\A_{m}(iy)}z) + e^{\varphi(iy)}\norm{\A(iy)\frac{z}{\norm{z}}}^{\beta}f(iy, \ol{\A(iy)}z) \Big|
	\end{align*}
	Can be bounded above by the sum of the following terms:
	\begin{eqnarray}
		&\hspace{2em}\abs{e^{\varphi(iy)}\norm{\A(iy)\frac{z}{\norm{z}}}^{\beta}-e^{\varphi_{m}(iy)}\norm{\A_{m}(iy)\frac{z}{\norm{z}}}^{\beta}} \cdot \abs{f(iy,\ol{\A(iy)}z) - f(ix,\ol{\A(ix)}z)} \label{eq:convergence1}\\
		&\abs{f(ix,\ol{\A(ix)}z)} \times \Big| -e^{\varphi(ix)}\norm{\A(ix)\frac{z}{\norm{z}}}^{\beta} + e^{\varphi_{m}(ix)}\norm{\A_{m}(ix)\frac{z}{\norm{z}}}^{\beta} \label{eq:convergence2} \\
		&\hspace{15em}- \left( -e^{\varphi(iy)}\norm{\A(iy)\frac{z}{\norm{z}}}^{\beta} + e^{\varphi_{m}(iy)}\norm{\A_{m}(iy)\frac{z}{\norm{z}}}^{\beta}\right) \Big| \nonumber \\
		&e^{\varphi_{m}(iy)}\norm{\A_{m}(iy)\frac{z}{\norm{z}}}^{\beta} \times \Big| f(iy,\ol{A(iy)}z) - f(iy, \ol{\A_{m}(iy)}z)\label{eq:convergence3} \\
		&\hspace{18em} -\left(  f(ix,\ol{A(ix)}z) - f(ix, \ol{\A_{m}(ix)}z)  \right) \Big| \nonumber \\
		&\abs{e^{\varphi_{m}(ix)}\norm{\A_{m}(ix)\frac{z}{\norm{z}}}^{\beta} -e^{\varphi_{m}(iy)}\norm{\A_{m}(iy)\frac{z}{\norm{z}}}^{\beta} } \label{eq:convergence4} \\
		&\times \abs{f(ix, \ol{\A_{m}(ix)}z) - f(ix, \ol{\A(ix)}z)} \nonumber
	\end{eqnarray}
	Let $k\geq 1$. Suppose that $x_{i}=y_{i}$ for $0\leq i \leq k-1$. First we bound \eqref{eq:convergence1}. Notice that
	
	\begin{align*}
		&\abs{e^{\varphi(iy)}\norm{\A(iy)\frac{z}{\norm{z}}}^{\beta}-e^{\varphi_{m}(iy)}\norm{\A_{m}(iy)\frac{z}{\norm{z}}}^{\beta}} \\
		&\leq  \norm{e^{\varphi(x)}\norm{\A(x)\frac{z}{\norm{z}}}^{\beta}-e^{\varphi_{m}(x)}\norm{\A_{m}(x)\frac{z}{\norm{z}}}^{\beta}}_{\infty} \\
		&\leq C\kappa_{m}
	\end{align*}
	Thus there exists a constant $C>0$ such that for all $k\geq 1$ and all  $x,y$ with $x_{i}=y_{i}$ for $0\leq i \leq k-1$ we have that \eqref{eq:convergence1} can be bounded above by $C\norm{f}_{\B(U)}\kappa_{k+1}\kappa_{m}$.
	
	Next we bound \eqref{eq:convergence2}. If $k+1\geq m$ then 
	\begin{align*}
		&\Big| -e^{\varphi(ix)}\norm{\A(ix)\frac{z}{\norm{z}}}^{\beta} + e^{\varphi_{m}(ix)}\norm{\A_{m}(ix)\frac{z}{\norm{z}}}^{\beta}\\
		&\hspace{5em}- \left( -e^{\varphi(iy)}\norm{\A(iy)\frac{z}{\norm{z}}}^{\beta} + e^{\varphi_{m}(iy)}\norm{\A_{m}(iy)\frac{z}{\norm{z}}}^{\beta}\right) \Big| \\
		&= \abs{e^{\varphi(iy)}\norm{\A(iy)\frac{z}{\norm{z}}}^{\beta} - e^{\varphi(ix)}\norm{\A(ix)\frac{z}{\norm{z}}}^{\beta}}\\
		&\leq  \abs{e^{\varphi(iy)}\norm{\A(iy)\frac{z}{\norm{z}}}^{\beta} - e^{\varphi(iy)}\norm{\A(ix)\frac{z}{\norm{z}}}^{\beta}}+ \abs{e^{\varphi(iy)}\norm{\A(ix)\frac{z}{\norm{z}}}^{\beta} - e^{\varphi(ix)}\norm{\A(ix)\frac{z}{\norm{z}}}^{\beta}}\\
		&\leq e^{\varphi(iy)}\abs{\norm{\A(iy)\frac{z}{\norm{z}}}^{\beta} -\norm{\A(ix)\frac{z}{\norm{z}}}^{\beta}} + \norm{\A(ix)\frac{z}{\norm{z}}}^{\beta} \abs{e^{\varphi(iy)}- e^{\varphi(ix)}}\\
		&\leq C\kappa_{k+1}.
	\end{align*}
	As $k+1\geq m$ we have that
	\begin{align*}
		\kappa_{k+1}&= \frac{\kappa_{k+1}}{\kappa_{k}}\kappa_{k}\\
		&= e^{-c[(k+1)^{p} - k^{p}]}\kappa_{k}\\
		&\leq  e^{-cpk^{p-1}}\kappa_{k}\\
		&\leq e^{-cp(m-1)^{p-1}}\kappa_{k}
	\end{align*}
	for some $C>0$. If $k+1<m$ then 
	\begin{align*}
		&\Big| -e^{\varphi(ix)}\norm{\A(ix)\frac{z}{\norm{z}}}^{\beta} + e^{\varphi_{m}(ix)}\norm{\A_{m}(ix)\frac{z}{\norm{z}}}^{\beta}\\
		&\hspace{5em}- \left( -e^{\varphi(iy)}\norm{\A(iy)\frac{z}{\norm{z}}}^{\beta} + e^{\varphi_{m}(iy)}\norm{\A_{m}(iy)\frac{z}{\norm{z}}}^{\beta}\right) \Big| \\
		&\leq 2\norm{ e^{\varphi_{m}(x)}\norm{\A_{m}(x)\frac{z}{\norm{z}}}^{\beta}-e^{\varphi(x)}\norm{\A(x)\frac{z}{\norm{z}}}^{\beta}}_{\infty}\\
		&\leq 2 C \kappa_{m}.
	\end{align*}
	As $k+1<m$ we have that
	\begin{align*}
		\kappa_{m}&=\frac{\kappa_{m}}{\kappa_k}\kappa_{k}\\
		&\leq  \frac{\kappa_{m}}{\kappa_{m-1}}\kappa_{k}\\
		&\leq  e^{-c[m^{p}-(m-1)^{p}]}\kappa_{k}\\
		&\leq e^{-cp(m-1)^{p-1}}\kappa_{k}.
	\end{align*}
	Thus there exists a constant $C>0$ such that for all $k\geq 1$ and all  $x,y$ with $x_{i}=y_{i}$ for $0\leq i \leq k-1$ we have that \eqref{eq:convergence2} can be bounded above by $C\norm{f}_{\B(U)}\kappa_{k}e^{-cp(m-1)^{p-1}}$.
	
	Next we bound \eqref{eq:convergence3}. Suppose that $k+1\geq m$ then
	\[ \ol{\A_{m}(iy)}=\ol{\A_{m}(ix)}  \]
	and thus
	\begin{align*}
		&\abs{ f(iy,\ol{\A(iy)}z) - f(iy, \ol{\A_{m}(iy)}z) -\left(  f(ix,\ol{A(ix)}z) - f(ix, \ol{\A_{m}(ix)}z)  \right)}\\
		&\leq \abs{ f(iy,\ol{\A(iy)}z) - f(ix,\ol{\A(ix)}z)} +\abs{f(ix, \ol{\A_{m}(iy)}z)-f(iy, \ol{\A_{m}(iy)}z)}\\
		&\leq \norm{f}_{\B(U)}(C_{K}+ 1)\kappa_{k+1} + \norm{f}_{\B(U)}\kappa_{k+1}
	\end{align*}
	If on the other hand $k+1<m$ then 
	\begin{align*}
		&\abs{ f(iy,\ol{\A(iy)}z) - f(iy, \ol{\A_{m}(iy)}z) -\left(  f(ix,\ol{A(ix)}z) - f(ix, \ol{\A_{m}(ix)}z)  \right)}\\
		&\leq 2\norm{  f(x,\ol{\A(x)}z) - f(x, \ol{\A_{m}(x)}z)}_{\infty}\\
		&\leq 2 C  \norm{f}_{\B(U)}\sup_{x}\norm{\ol{\A(x)}-\ol{\A_{m}(x)}}\\
		&\leq 2 C \norm{f}_{\B(U)} \var_{m}\A \\
		&\leq 2 C \norm{f}_{\B(U)} \abs{\A}_{\set{\kappa_{n}}} \kappa_{m}
	\end{align*}
	Following the same argument from bounding \eqref{eq:convergence2} we find that there exists a constant $C>0$ such that for all $k\geq 1$ and all  $x,y$ with $x_{i}=y_{i}$ for $0\leq i \leq k-1$ we have that \eqref{eq:convergence3} can be bounded above by $C\norm{f}_{\B(U)}\kappa_{k}e^{-cp(m-1)^{p-1}}$.
	
	Finally we bound \eqref{eq:convergence4}. First notice that if $k+1\geq m$ then 
	\[ \ol{\A_{m}(iy)}=\ol{\A_{m}(ix)}  \text{ and } \varphi_{m}(ix)=\varphi(iy) \]
	thus \eqref{eq:convergence4} is $0$. If $k+1<m$ then 
	\begin{align*}
		\abs{f(ix, \ol{\A_{m}(ix)}z) - f(ix, \ol{\A(ix)}z)} \leq C \norm{f}_{\B(U)} \abs{\A}_{\set{\kappa_{n}}} \kappa_{m}
	\end{align*}
	Following the same argument from bounding \eqref{eq:convergence2} we find that there exists a constant $C>0$ such that for all $k\geq 1$ and all  $x,y$ with $x_{i}=y_{i}$ for $0\leq i \leq k-1$ we have that \eqref{eq:convergence4} can be bounded above by $C\norm{f}_{\B(U)}\kappa_{k}e^{-cp(m-1)^{p-1}}$.
	
	Putting these together we find that there exists a constant $C$ such that
	\[ \abs{\L_{m,\beta}f-\L_{\beta}f}_{\set{\kappa_{n}}} \leq C \norm{f}_{\B(U)} e^{-cp(m-1)^{p-1}}. \]
	Next we need to bound the uniform norm. Let $x  \in \S_{T}^{+}$ and $z \in U$ then
	\begin{align*}
		&\abs{\L_{m,\beta}f(x,z) - \L_{\beta}f(x,z)}\\
		&\leq \sum_{i:ix \in \S_{T}^{+}} \abs{ e^{\varphi_{m}(ix)}\norm{\A_{m}(ix)\frac{z}{\norm{z}}}^{\beta}f(ix, \ol{\A_{m}(ix)}z) - e^{\varphi(ix)}\norm{\A(ix)\frac{z}{\norm{z}}}^{\beta}f(ix, \ol{\A(ix)}z) }
	\end{align*}
	Notice that
	\begin{align*}
		&\abs{ e^{\varphi_{m}(ix)}\norm{\A_{m}(ix)\frac{z}{\norm{z}}}^{\beta}f(ix, \ol{\A_{m}(ix)}z) - e^{\varphi(ix)}\norm{\A(ix)\frac{z}{\norm{z}}}^{\beta}f(ix, \ol{\A(ix)}z) }\\
		&\leq \abs{f(ix, \ol{\A_{m}(ix)}z)}\abs{ e^{\varphi_{m}(ix)}\norm{\A_{m}(ix)\frac{z}{\norm{z}}}^{\beta}- e^{\varphi(ix)}\norm{\A(ix)\frac{z}{\norm{z}}}^{\beta} }\\
		&\hspace{5em}+e^{\varphi(ix)}\norm{\A(ix)\frac{z}{\norm{z}}}^{\beta}\abs{ f(ix, \ol{\A_{m}(ix)}z) - f(ix, \ol{\A(ix)}z) }\\
		&\leq \norm{f}_{\B(U)}\abs{ e^{\varphi_{m}(ix)}\norm{\A_{m}(ix)\frac{z}{\norm{z}}}^{\beta}- e^{\varphi(ix)}\norm{\A(ix)\frac{z}{\norm{z}}}^{\beta} }\\
		&\hspace{5em}+e^{\varphi(ix)}\norm{\A(ix)\frac{z}{\norm{z}}}^{\beta}\norm{f}_{\B(U)}\norm{ \ol{\A_{m}(ix)}z - \ol{\A(ix)}z }\\
		&\leq C \norm{f}_{\B(U)} \kappa_{m}
	\end{align*}
	for some $C>0$. Putting all of this together we arrive at the conclusion that 
	\[ \norm{\L_{m,\beta} - \L_{\beta}}_{\B(U),\op} \leq C e^{-cp(m-1)^{p-1}}. \]
\end{proof}

\bibliography{MatrixTransferOpBib}{}
\bibliographystyle{abbrv}

\end{document}